\documentclass[a4paper,11pt]{article}
\usepackage[utf8]{inputenc}
\usepackage[T1]{fontenc}
\usepackage[english]{babel}
 
\usepackage[tbtags]{amsmath}
\usepackage{amssymb}
\usepackage{amsfonts}
\usepackage{amsthm}
\usepackage{calrsfs}
\usepackage{array}
\usepackage{bbm}
\usepackage{ulem}
\usepackage{upgreek}

\usepackage{mathtools}
\usepackage[svgnames]{xcolor}
\usepackage{hyperref}

\numberwithin{equation}{section}
\renewcommand{\aa}{{a_{n,1}}}

\newcommand{\bb}{{a_{n,2}}}

\newcommand\cc{{C}}

\newcommand\dd{D}

\newcommand\dv{\mathop{\rm div}}

\renewcommand{\k}{k}

\newcommand{\m}[1]{\mathbbm{#1}}

\newcommand{\q}[1]{\mathcal{#1}}

\newcommand\san{{s_{A,n}}}

\newcommand\SCinq{1}

\DeclareMathOperator{\Id}{\mathrm{Id}}

\theoremstyle{plain}
\newtheorem{thm}{Theorem}
\newtheorem*{thm*}{Theorem}

\newtheorem{propo}[thm]{Proposition}
\newtheorem{prop}{Proposition}[section]

\newtheorem{lem}[prop]{Lemma}

\newtheorem{cl}[prop]{Claim}

\theoremstyle{definition}

\theoremstyle{remark}
\newtheorem*{nb}{Remark}

\makeatletter
\def\blfootnote{\xdef\@thefnmark{}\@footnotetext}
\makeatother

\title{\textbf{Behavior rigidity near non-isolated blow-up points for
  the semilinear heat equation}}
\author{Frank Merle\\
{\it \small CY Cergy Paris Universit\'e
and IHES}\\
Hatem Zaag\\
{\it \small 
Universit\'e Sorbonne Paris Nord,
  LAGA, CNRS (UMR 7539), F-93430, Villetaneuse, France}
}
\allowdisplaybreaks

\begin{document}     












  
\newpage

\maketitle   

\begin{abstract}
We consider the semilinear heat equation with Sobolev subcritical power nonlinearity in
dimension $N=2$, and $u(x,t)$ a solution which blows up in finite time $T$.
Given a non isolated blow-up point $a$, we assume that the Taylor
expansion of the solution near $(a,T)$ obeys some degenerate
situation labeled by some even integer $m(a)\ge 4$. If we have a
sequence $a_n \to a$
as $n\to \infty$, we show after a change of coordinates and the
extraction of a subsequence that either  $\aa-a_1 = o((a_{n,2}-a_2)^2)$ or
$|a_{n,1}-a_1||a_{n,2}-a_2|^{-\beta} |\log|a_{n,2}-a_2||^{-\alpha} \to L>
0$ for some $L>0$,
where $\alpha$ and $\beta$ 
enjoy a finite number of rational values with $\beta \in(0,2]$ and $L$ is
a solution of a polynomial equation depending on the
coefficients of the Taylor expansion of the solution.
If $m(a)=4$, then $\alpha=0$ and either $\beta=3/2$ or $\beta =2$.
\end{abstract}

\medskip



\medskip




\section{Introduction}
We consider the semilinear heat equation in space dimension $N\ge 1$
with a Sobolev subcritical power nonlinearity:
\begin{equation}\label{equ}
\left\{
\begin{array}{l}
\partial_t u =\Delta u+|u|^{p-1}u,\\
u(0)=u_0\in L^\infty(\m R^N)
\end{array}
\right.
\end{equation}
where
\begin{equation}\label{condp}
p>1,\;\;(N-2)p<N+2.
\end{equation}
Other phenomena arise in the Sobolev critical and supercritical cases (see from
example Merle, Rapha\"el and Szeftel \cite{MRSimrn20}, Schweyer \cite{Sjfa12} and the
references therein).

\medskip

We consider $u(t):x\in{\m R}^N \rightarrow u(x,t)\in{\m R}$ 
a solution which blows up at time $T>0$ and introduce the set of its blow-up points
\[
\q S = \{a\in \m R^N\;\;|\;\;|u(a,t)|\to +\infty\mbox{ as }t\to T\}
\]
(note that $u(x,t) \to u^*(x)$ as $t\to T$ whenever $x\not \in \q
S$). In the literature, we know examples of blow-up solutions where
$\q S$ is finite or the union of some concentric spheres.

\medskip

In this paper, we are interested in non-isolated blow-up points.
No results are available in this context except for curve
singularity (see Zaag \cite{Zihp02, Zcmp02, Zbeit00, Zdmj06} and
Ghoul, Nguyen and Zaag \cite{GNZans17}). Our goal is to introduce new techniques to track this
kind of question.

\medskip

When $N=1$, we know from Chen and Matano \cite{CMjde89} and Herrero and Vel\'azquez \cite{HVcpde92} that
all the blow-up points are isolated.
As we are interested in the asymptotic behavior
near non-isolated blow-up points, we need to assume that $N\ge 2$.

\medskip

Let  us note that all our statements do hold for unsigned solutions
(Proposition \ref{propexp} together with Theorems \ref{th0m} and
\ref{cor0m}). However, for simplicity, we reduce in the presentation
and the proofs to the case of nonnegative solutions,
without loss of generality. Indeed, we know from Corollary 2 page 108
in Merle and Zaag \cite{MZma00} that the solution has a constant sign
in some neighborhood of any given blow-up point, which means that in
similarity variables \eqref{defw}, the
unsigned case can be seen as a perturbation of the nonnegative case
with arbitrarily small exponential terms. Although the proof
needs a crucial blow-up criterion given below in Proposition
\ref{problo} and which
is
valid only for nonnegative
solutions, one should
keep in mind that 
we have a twin version
valid
for unsigned solutions
and given in Proposition 1.2 page 111 in \cite{MZma00}.


\medskip

Given  $a\in \q S$, it is convenient to study the local behavior of $u(x,t)$ near $(a,T)$ in the similarity variables version $w_a(y,s)$ first introduced by Giga and Kohn in \cite{GKcpam85} by
\begin{equation}\label{defw}
w_a(y,s) = (T-t)^{\frac 1{p-1}}u(x,t) \mbox{ where } y = \frac{x-a}{\sqrt{T-t}}\mbox{ and }s=-\log(T-t).
\end{equation}
Using \eqref{equ}, we see that $w_a$ (or $w$ for short) satisfies the following PDE for all $s\ge -\log T$ and $y\in \m R^N$:
\begin{equation}\label{eqw}
\partial_s w = \Delta w-\frac 12 y \cdot \nabla w -\frac w{p-1} +|w|^{p-1}w.
\end{equation}
From Giga and Kohn \cite{GKcpam89}, we know that
\begin{equation}\label{conv}
w_a(y,s) \to \kappa\equiv(p-1)^{-\frac 1{p-1}}\mbox{ as }s\to \infty,
\end{equation}
uniformly on compact sets and also in $L^2_\rho(\m R^N)$, the $L^2$ space with respect to the measure density
\begin{equation}\label{defro}
\rho(y) = \exp\left(-\frac{|y|^2}4\right)/(4\pi)^{N/2}.
\end{equation}
According to Vel\'azquez \cite{Vtams93} (see also Filippas and Kohn
\cite{FKcpam92} together with Filippas and Liu \cite{FLihp93}) , we may refine that
convergence and obtain the following first order classification:\\
- either 
\begin{equation}\label{A2}
  w_a(y,s) -\kappa\sim -\frac \kappa{4ps}\sum_{i=1}^l h_2(y_l)
\end{equation}
where $l=1,\dots,N$, after a rotation of coordinates;\\
- or 
\begin{equation}\label{Am}
  w_a(y,s) - \kappa \sim e^{-(\frac m2-1)s}
\sum_{j_1+\dots+j_N=m}C_{m,j_2,,\dots,j_N}h_{j_1}(y_1)\dots h_{j_N}(y_N)
\end{equation}
as $s\to \infty$, for some even integer $m=m(a)\ge 4$, where $y=(y_1,\dots,y_N)$, $h_j(\xi)$ is the rescaled Hermite polynomial defined by 
\begin{equation}\label{defhj}
h_j(\xi) = \sum_{i = 0}^{\big[j/2\big]} \frac{j!}{i!(j - 2i)!}(-1)^i\xi^{j - 2i},
\end{equation}
and the multilinear form 
\begin{equation}\label{defmult}
  \sum_{j_1+\dots+j_N=m}C_{m,j_2,\dots,j_N}y_1^{j_1}\dots y_N^{j_N}
\end{equation}
 is non zero and nonpositive.
Extending the definition of $m(a)$ by $2$ if \eqref{A2} holds, Khenissy, Rebai and Zaag
called $m(a)$ the ``profile order at $a$'' in 
\cite{KRZihp11}.

\medskip

From Vel\'azquez \cite{Vcpde92}, we know that this expansion may
indicate whether $a$ is an isolated blow-up point or not. Indeed, from Theorem 2
page 1570 in \cite{Vcpde92}, we have the following:

\medskip

(i) {\it If \label{resvel} \eqref{A2} holds with $l=N$ or \eqref{Am}
  holds with the multilinear form in \eqref{defmult}
  negative, then $a$ is an isolated blow-up point}.

(ii) {\it More generally,  for any
    $\epsilon>0$, there exists $\delta>0$ such that $\q S \cap
    B(0,\delta) \subset \Omega_\epsilon$ where:\\
    - $\Omega_\epsilon\equiv
    \{y\in \m R^N\;|\;\sum_{i=1}^l y_i^2 \le  \epsilon|y|^2\}$ if \eqref{A2} holds,\\
    -     $\Omega_\epsilon \equiv \{y\in \m R^N\;|\;
    \left|\sum_{j_1+\dots+j_N=m}C_{m,j_2,\dots,j_N}y_1^{j_1}\dots y_N^{j_N}\right|
    \le \epsilon|y|^m\}$} if \eqref{Am} holds.

\medskip

\noindent Expressing item (ii) differently, we may say that the set
$\Omega_0\cap\{|y|=1\}$ (which is finite) gives indications on the
location of neighboring blow-up points. Indeed, if $\Omega_0\cap \{|y|=1\}=\emptyset$, then
we are in Case (i), and $a$ is an isolated blow-up point ; if
$\Omega_0\cap \{|y|=1\}\neq \emptyset$, then we cannot assert whether
$a$ is isolated or not (in fact, we believe the converse of item (i) to be very
hard); if we further assume that $a$ is non isolated, then for sure the neighboring blow-up points are
located ``along'' the directions of the non-zero elements of $\Omega_0\cap \{|y|=1\}$.

\bigskip

In this paper, our first goal is to refine  the expansion
\eqref{A2}-\eqref{Am} up to the second order, and more if possible.
In fact, we don't consider the case \eqref{A2}, where more refined expansions were obtained in a
series of papers (see Zaag \cite{Zihp02, Zcmp02, Zbeit00} and Ghoul,
Nguyen and Zaag \cite{GNZans17}).
We will instead focus on the case \eqref{Am}, where no refinement is available, up
to our knowledge. In addition, the exponential decay observed in the
case \eqref{Am} is more advantageous than the polynomial interaction of
case \eqref{A2}, which allows us to better handle the interactions between
the various components of the solution in Theorem \ref{cor0m} below.

\medskip

Our first
result states that in fact such an expansion is possible up to any order:
\begin{propo}[Asymptotic expansion in similarity variables for $m\ge 4$]\label{propexp}
Consider $w(y,s)$ a solution to equation \eqref{eqw} defined for all
$y\in \m R^N$ and $s\ge s_0$ for some $s_0$ and assume that it
satisfies the expansion \eqref{Am} for some even integer
$m\ge 4$.
Then, for any integer $M\ge 2m$,
\[
  w(y,s) = \kappa+
  \sum_{
{\scriptsize \begin{array}{l}
    j=2m,\dots,M\\
    l=1,\dots,\alpha_j\\
    (i_1,\dots,i_N)\in E_{j,l}
\end{array}}
}
 e^{-\frac j2 s}s^l
    h_{i_1}(y_1)\dots h_{i_N}(y_N) +o(e^{-\frac M2s}s^{\alpha_{M}})
\]
as $s\to \infty$, uniformly on compact sets and in $L^q_\rho$ for any
$q\ge 2$, where $\alpha_j \in \m N$ and $E_{j,l}\subset \m N^N$ is finite. 
\end{propo}
If the multilinear form in \eqref{defmult} has some degenerate
directions, assuming furthermore that $w=w_a$ where $a$ is a non
isolated blow-up point of \eqref{equ},
we may uncover some rigidity in the expansion, in the
sense that we show that some coefficients of the Taylor expansion are zero.
In order to simplify the presentation, we assume in the following that
\begin{equation}\label{simple}
  N=2,\;\;p=2,\mbox{ and } u_0\ge 0.
\end{equation}
The general case where $N=2$, $p>1$ and $u_0$ has no sign follows with
the same proof and natural adaptations of the statements.
More precisely, we claim the following:
\begin{thm}[Second order refined asymptotic expansion near a non
  isolated point, when
 $m\ge 4$]\label{th0m}
 Consider $u(x,t)$ a
   solution of equation \eqref{equ} blowing up at
  time $T>0$.
 Assume in addition that the origin is a
 non isolated blow-up point with $m(0)=m\ge 4$.
 Then:\\
 (i) Up to some rotation of coordinates,
it holds that
\begin{align}
 w _0(y,s)=1+& e^{(1-\frac m2)s}\sum_{j=0}^{m-2}C_{m,j}h_{m-j}(y_1)h_{j}(y_2) \label{do}\\
           &+e^{\frac
             {1-m}2s}\sum_{j=0}^mC_{m+1,j}h_{m+1-j}(y_1)h_{j}(y_2)
            \nonumber
              +O(se^{-\frac m2 s})
\end{align}
as $s\to \infty$ in $L^q_\rho$ for any $q\ge 2$,  for some real coefficients $C_{i,j}$ such that the
multilinear form in \eqref{defmult} is nonpositive, where $w_0(y,s)$
is defined in \eqref{defw}.\\
(ii) If $m=4$, then
\begin{equation}\label{discriminant}
  (C_{4,0}, C_{4,1}, C_{4,2}) \neq (0,0,0),\;\;
C_{4,0}\le 0,\;C_{4,2}\le 0\mbox{ and }C_{4,1}^2-4C_{4,0}C_{4,2}\le 0.
\end{equation}
\end{thm}
\begin{nb} If the origin is an isolated blow-up point, we expect no
  rigidity in the coefficients of the Taylor expansion.
  \end{nb}
\begin{nb}
  Several higher order improvements of
\eqref{do}  (showing cancelations of coefficients)
  are available in
  the proof.
\end{nb}
Our second statement concerns the local geometry of the blow-up set:
\begin{thm}[Rigidity in the blow-up set
  near a non-isolated blow-up point when
$m\ge 4$]\label{cor0m}
 Consider under the hypotheses of Theorem \ref{th0m} some sequence
  $a_n=(\aa,\bb)$ of non-zero blow-up points converging to the
  origin. Then,\\
  (i) It holds that
  \[
a_n \cdot e^\bot = o(a_n \cdot e)\mbox{ as }n\to \infty,
\]
up to extracting a subsequence, where $e\in \m R^2$ is a unitary
vector of a degenerate direction of the multilinear form
\eqref{defmult}, $e^\bot$ is unitary and $e\cdot e^\bot=0$.\\
(ii) Up to a rotation and a symmetry of the axes, and up
to a subsequence, it holds that
$\aa \ge 0$, $\bb\ge 0$ and \eqref{do} still holds with possibly
different constants, with
\[
 \mbox{either }\aa = o(\bb^2)\mbox{ or }\aa\sim L|\bb|^\beta|\log \bb|^\alpha,
\]
for some $\beta$ and $\alpha$ enjoying a finite number of rational
values with $0<\beta\le 2$ (see the proof for a finer description of
the localization of $\beta$).\\
Moreover,
$L$ is
a solution of a polynomial equation involving the coefficients of the
Taylor expansion \eqref{do} of $w_0(y,s)$ or one of its higher order
refinements.\\
(iii) When $m=4$, the only possibilities for $e$ in item (i) are
$e=(0,1)$ and
$e=(-2C_{4,2},C_{4,1})$,
the second possibility
occurring only if $C_{4,1}^2-4C_{4,0}C_{4,2}=0$ and $C_{4,2}\neq 0$. In addition, we have
the following simple statement for the behavior of
$a_n$ in item (ii): 
\[
  \mbox{either }\aa = o(\bb^2),
  \mbox{ or }\aa \sim L \bb^2
  \mbox{ or } \aa \sim L \bb^{3/2}\mbox{ with }L>0.
\]
\end{thm}
\begin{nb}
  The possible values taken by $\beta$ are such that
  $\frac{\beta-1}2\in E_1\cup E_2$ where the 2 sets $E_1$ and $E_2$
  are defined respectively in \eqref{defE1} and \eqref{defE2} below in
  the proof.
\end{nb}
\begin{nb}
Our strategy to find the precise subquadratic regimes when $m=4$  can
be carried out for any fixed $m\ge 6$. For example, when $m=6$, we found the
following values for $\beta$ :   $4/3$, $3/2$ and $5/3$, with
$\alpha=0$ each time.
 \end{nb}
\begin{nb}
Note that in  the case $m=2$, we
think that a similar result also holds. However, the proof should be
very different, since we are in polynomial scales of time $s$ (see
\eqref{A2}), unlike the case $m\ge 4$ where we are in exponential
scales (see \eqref{Am}).
\end{nb}
\begin{nb}
We believe that our techniques should apply in the case $N\ge 3$ and
yield a similar result.
\end{nb}

  \bigskip

  Let us briefly explain our strategy in this paper. Using Proposition
  \ref{propexp}, we first make a Taylor expansion of $w_0$. Given $b$
  a nearby blow-up point, we derive from this a Taylor expansion for
  $w_b$ (use \eqref{wbw000} given below). Since $w_b$ is uniformly
bounded in $L^\infty$, its components in the expansion \eqref{defvij}
cannot grow, which implies some cancelations in the Taylor
coefficients of $w_0$, justifying \eqref{do}. In a second step, by the same argument, we
derive some constraints on the location of the neighboring blow-up
points, leading to Theorems \ref{th0m} and \ref{cor0m}.

\medskip

Note that for simplicity, we give the proofs only under assumption \eqref{simple}.
The general case follows by the same proof. The only delicate point is
to replace the blow-up criterion given below in Proposition \ref{problo} by its 
 twin version valid for unsigned solutions and given in Proposition
 2.1 page 111  of Merle and Zaag \cite{MZma00}.

\section{Existence of an expansion of the solution in similarity variables up to any order}\label{secexp}
We prove Proposition \ref{propexp} in this section. 
\begin{proof}[Proof of Proposition \ref{propexp}]
Note
that the case $m= 4$ is harder
than the case $m\ge 6$. Indeed, the quadratic term in the equation
induces more interactions in the former than in the latter case. For
that reason, we only give the proof in the harder case, namely when $m=4$.

\medskip

Introducing
\begin{equation}\label{defv}
v=w-1,
\end{equation}
 we see from equation \eqref{eqw} that $v$ satisfies the following PDE for all $s\ge -\log T$ and $y\in \m R^2$:
\begin{equation}\label{eqv}
\partial_s v = \q L v +v^2
\end{equation}
where 
\begin{equation}\label{defL}
\q L v = \Delta v -\frac 12 y\cdot \nabla v+v= \frac 1\rho\dv(\rho \nabla v)+v,
\end{equation}
and $\rho(y)$ is defined in \eqref{defro}.
The operator $\q L$ is self-adjoint in $L^2_\rho(\m R^2)$, and its spectrum is given by the set $\{1-\frac j2\;|\;j\in \m N\}$, which consists only in eigenvalues, having  $h_n(y_1)h_l(y_2)$ \eqref{defhj} as eigenfunctions, in the sense that
\[
\q L(h_n(y_1)h_l(y_2))= \left(1-\frac{n+l}2\right)h_n(y_1)h_m(y_2).
\]
Expanding $v$  as follows:
\begin{equation}\label{defvij}
v(y,s) = \sum_{i\in \m N,\;0\le j\le i}v_{i,j}(s) h_{i-j}(y_1) h_j(y_2)
\end{equation}
and
introducing for any $l\in \m N$,
$v_{-,l}$ and the $L^2_\rho(\m R^2)$ projector $P_{-,l}$ defined by
\begin{equation}\label{defv-m}
v_{-,l}(y,s) = 
P_{-,l} (v)\equiv \sum_{i\ge l,\;0\le j\le i}v_{i,j}(s) h_{i-j}(y_1) h_j(y_2),
\end{equation}
we write the following equations satisfied by $v_{i,j}$ and $v_{-,l}$:
\begin{align}
v_{i,j}'(s)&=\left(1-\frac i2\right)v_{i,j}(s)+\int v(y,s)^2 k_{i-j}(y_1)k_j(y_2) \rho(y)dy, \label{eqvij}\\
\partial_s v_{-,l} &=\q L v_{-,l}+P_{-,l}(v^2),\label{eqv-m}
\end{align}
where
\begin{equation}\label{defk}
k_j(\xi) = h_j(\xi)/\|h_j\|_{L^2_\rho}^2.
\end{equation}
This way, we are in a position to perform the first iteration, in order to refine the asymptotic expansion in \eqref{Am}.

\medskip

{\bf Part 1: The first iteration}

Like the following iterations, we proceed in 3 steps in order to get the next terms in the expansion, starting from \eqref{Am}:\\
- we first use parabolic regularity to improve the convergence in \eqref{Am}, from $L^2_\rho(\m R^2)$ to $L^q_\rho(\m R^2)$ for any $q\ge 2$;\\
- then, we use the improved convergence to expand the quadratic term in equation \eqref{eqv}, and write an ODE satisfied by the component $v_{i,j}$ defined in \eqref{defvij}; solving that ODE gives an estimate on $v_{i,j}$, better than what \eqref{Am} states;\\
- finally, we use again the estimate on the quadratic term of equation \eqref{eqv} and write from \eqref{eqv-m} a differential inequality satisfied by $v_{-,l}$ defined in \eqref{defv-m}; integrating that inequality gives an improved estimated on $v_{-,l}$.

\medskip

{\bf Step 1: Parabolic regularity to improve the convergence in \eqref{Am}}

The following parabolic regularity estimate
is crucial for the improvement:
\begin{lem}(Regularizing effect of the operator $\q L$)\label{lemVel}
$ $\\
(i) {(\bf Herrero and Vel\'azquez \cite{HVihp93})} For any $q>1$, $r>1$, $\psi_0 \in L^q_\rho(\m R^N)$ and $s>\max\left(0, -\log(\frac{q-1}{r-1})\right)$, it holds that
\[
\|e^{\q Ls}\psi_0\|_{L^r_\rho(\m R^N)}\le  \frac{C(q,r)e^se^{-\frac {Ns}{2r}}}{(1-e^{-s})^{\frac N{2q}}(q-1-e^{-s}(r-1))^{\frac N{2r}}} \|\psi_0\|_{L^q_\rho(\m R^N)}.
\]
(ii) Consider $r\ge 2$ and $v_0\in L^r_\rho(\m R^N)$ such that $|v_0(y)|+|\nabla v_0(y)|\le C(1+|y|^k)$ for some $k\in \m N$. Then, for all $s\ge 0$, we have
$\|e^{\q L s}v_0\|_{L^r_\rho(\m R^N)}\le e^s\|v_0\|_{L^r_\rho(\m R^N)}$.
\end{lem}
\begin{nb} Although we are working in two space dimensions, we felt it better to state the result for any $N\ge 1$, for future purpose.
\end{nb}
\begin{proof}
$ $\\
(i) See Section 2 page 139 of \cite{HVihp93}. Although the proof in that paper was given for $N=1$, it extends to higher dimensions with no difficulty.\\ 
(ii) From Lemma 4 page 555 in Bricmont and Kupiainen, we know that $|e^{\q Ls}(1+|y|^k)|\le Ce^s(1+|y|^k)$. Therefore, if $\psi(y,s) = e^{\q Ls}v_0(y)$, we see from the hypotheses that  $|\psi(y)|+|\nabla \psi(y)|\le Ce^s(1+|y|^k)$. Using the linear equation satisfied by $\psi$, we justify that 
\begin{align*}
\frac d{ds} \int|\psi(y,s)|^r\rho(y) dy &= - r(r-1)\int|\psi(y,s)|^{r-2}|\nabla \psi(y,s)|^2 \rho(y) dy+\int|\psi(y,s)|^r\rho(y) dy\\
& \le \int|\psi(y,s)|^r\rho(y) dy,
\end{align*}
and the result follows.
\end{proof}
Now, if 
\begin{equation}\label{defvbar}
\bar v(y,s) =e^{-s}\sum_{j=0}^4 C_{4,j}h_{4-j}(y_1)h_j(y_2)\mbox{ and } g= v-\bar v
\end{equation}
we see from \eqref{Am} and the definition \eqref{defv} of $v(y,s)$ that 
\begin{equation}\label{gl2}
\|g(s)\|_{L^2_\rho} =o(e^{-s})\mbox{ as }s\to \infty.
\end{equation}
Moreover, $\bar v$ is an approximate solution of equation \eqref{eqv}, up to some error term, in the sense that
\[
\partial_s \bar v = \q L \bar v+\bar v^2+O\left(e^{-2s}(1+|y|^8)\right)
\]
(in fact, in this first iteration, unlike the following iterations, $\bar v$ solves the linear equation, but we felt it better to write that equation as a perturbation of the nonlinear equation \eqref{eqv}, since this latter fact will be always true in the following iterations). Using \eqref{eqv}, we write the following equation satisfied by $g(y,s)$:
\begin{equation}\label{eqg}
\partial_s g = \q L g +(v+\bar v)g+O\left(e^{-2s}(1+|y|^8)\right).
\end{equation}
Recalling the following estimate from Giga and Kohn \cite{GKcpam89} and Giga, Matsui and Sasayama \cite{GMSiumj04}:
\[
\|u(t)\|_{L^\infty}\le \frac C{T-t},
\]
we see by definitions \eqref{defw} and \eqref{defv} of $w$ and $v$ that
\begin{equation}\label{b1}
\|v(s)\|_{L^\infty} \le C.
\end{equation}
Furthermore, by defintion \eqref{defvbar} of $\bar v$, we see that 
\begin{equation}\label{b2}
|\bar v(y,s)|\le C,\mbox{ if }|y|<e^{\frac s4}.
\end{equation}
Consider then $q \ge 2$. Using the Duhamel formulation of equation \eqref{eqg} together with \eqref{b1} and \eqref{b2}, we write for some $M>0$ and for some $s^*$ to be taken fixed large enough,
\begin{align*}
\|g(s)\|_{L^q_\rho} &\le \|e^{s^*(\q L+M\Id)}g(s-s*)\|_{L^q_\rho}
+\int_{s-s^*}^s \|e^{(s-\sigma)(\q L+M\Id)}1_{\{|x|>e^{\frac \sigma 4}\}}\bar vg(\sigma)\|_{L^q_\rho}d\sigma\\
&+C\int_{s-s^*}^s \|e^{(s-\sigma)(\q L+M\Id)}e^{-2\sigma}(1+|x|^8)\|_{L^q_\rho}d\sigma.
\end{align*}
Fixing $s^* > \log(q-1)$, we write from Lemma \ref{lemVel} 
\begin{equation}\label{boundg}
\|g(s)\|_{L^q_\rho} \le C\|g(s-s*)\|_{L^2_\rho}
+\int_{s-s^*}^s\|1_{\{|x|>e^{\frac \sigma 4}\}}\bar vg(\sigma)\|_{L^q_\rho}d\sigma+Ce^{-2s}\int_{s-s^*}^s \|(1+|x|^8)\|_{L^q_\rho}d\sigma.
\end{equation}
Note that $\|g(s-s*)\|_{L^2_\rho}=o(e^{-s})$ as $s\to \infty$ from \eqref{gl2} and $\|(1+|x|^8)\|_{L^q_\rho}\le C$. It remains then to estimate the middle term in the right-hand side of \eqref{boundg} in order to conclude.\\
Since for $\sigma\in [s-s^*, s]$ and $|y|>e^{\frac \sigma 4}$, we have by
definition \eqref{defvbar} of $g$ and $\bar v$ and estimate \eqref{b1}, 
\[
|\bar v(y,\sigma)g(y,\sigma)| =|\bar v(y,\sigma)(v(y,\sigma)-\bar v(y,\sigma)| 
\le Ce^{-s}(1+|y|^4)+Ce^{-2s}(1+|y|^8),
\]
and
\[
0\le \rho(y) \le Ce^{-|y|^2/8}\times e^{-|y|^2/8}\le Ce^{-e^{(s-s^*)/2}/8}e^{-|y|^2/8},
\]
it follows that 
\[
\|1_{\{|x|>e^{\frac \sigma 4}\}}\bar vg(\sigma)d\sigma\|_{L^q_\rho}\le Ce^{-s-e^{(s-s^*)/2}/8}.
\]
Gathering the above-mentioned bounds, it follows that
\[
\|g(s)\|_{L^q_\rho}= o(e^{-s})\mbox{ as }s\to \infty,
\]
which is precisely the goal of Step 1.



\medskip

{\bf Step 2: Refinement of the behavior of $v_{i,j}$}

Since estimate \eqref{Am} holds in $L^4_\rho$ by Step 1, we may use it to refine the ODE \eqref{eqvij} satisfied by $v_{i,j}$ and write:
\begin{equation}\label{ineqvij}
\left|v_{i,j}'(s)-\left(1-\frac i2\right)v_{i,j}(s)\right| \le C\|v(s)\|_{L^4_\rho}^2\le Ce^{-2s}.\\
\end{equation}
Since $|v_{i,j}(s)|\le Ce^{-s}$ by \eqref{defv} and \eqref{Am}, using
elementary ODE techniques, we derive for $s$ large enough:
\begin{equation}\label{exp1}
\sup_{i\le 5}|v_{i,j}(s)-C_{i,j}e^{(1-\frac i2)s}|\le Cse^{-2s},
\end{equation}
with
\[
  C_{i,j} = 0\mbox{ if }i\le 3.
\]
Note that the constant $C$ in this step may depend on $(i,j)$.

\medskip

{\bf Step 3: Refinement of the behavior of $v_{-,l}$}

Since $\|P_{-,l}(v^2)\|_{L^2_\rho}\le \|v^2\|_{L^2_\rho}=\|v\|_{L^4_\rho}^2\le Ce^{-2s}$ by the regularity in Step 1, 
we see from equation \eqref{eqv-m} that 
\begin{equation}\label{ineq1}
\|v_{-,l}\|_{L^2_\rho}'
\le \left(1-\frac l2\right)\|v_{-,l}\|_{L^2_\rho}+Ce^{-2s},
\end{equation}
where $C$ may depend on $l$. Taking $l=6$ and integrating this inequality yields
\begin{equation}\label{exp2}
\|v_{-,6}\|_{L^2_\rho}\le Cse^{-2s}.
\end{equation}
Gathering \eqref{exp1} and \eqref{exp2}, we see that
\begin{equation}\label{first-iteration}
v(y,s) =e^{-s}\sum_{j=0}^4 C_{4,j}h_{4-j}(y_1)h_j(y_2)+e^{-\frac 32s}\sum_{j=0}^{5} C_{5,j}h_{5-j}(y_1)h_j(y_2) +O(se^{-2s})
\end{equation}
as $s\to \infty$, in $L^2_\rho(\m R^2)$.
This gives the first terms of the expansion announced in Proposition \ref{propexp}.

\medskip

{\bf Part 2: the following iterations} 

The strategy developed in the first iteration works here, iteratively!
Let us see how the second iteration works: using the parabolic
regularity developed in Step 1 of the first iteration, one can show
that estimate \eqref{first-iteration} holds also in $L^q_\rho$, for
any $q\ge 2$, and also uniformly on compact sets. Using this, one can
refine the quadratic term in equations \eqref{eqvij} and \eqref{eqv-m}
up to $O(se^{-3s})$. Integrating those equations (with $i\le 7$ and
$l=8$), we refine the expansion of $v_{,j}$ and $v_{-,l}$ up to that
order, resulting  in an expansion of $v$ in $L^2_\rho$, up to $O(s^2
e^{-3s})$, and showing terms like $e^{-s}$, $e^{-\frac 32s}$,
$se^{-2s}$ and $e^{-2s}$. Iterating the process, we may get an
expansion for $v$ valid up to any order, in any $L^q_\rho$ space with
$q\ge 2$, and also uniformly on compact sets. This concludes the proof
of Proposition \ref{propexp}, when $p=2$, the solution is nonnegative
and $m=4$. As we explained in the beginning of the proof, this is the
harder case, and the adaptation to the general case is straightforward.
\end{proof}
\section{Rigidity in the Taylor expansion
  for general $m\ge 4$}\label{sectaylor}
 This section is devoted to the proof of Theorem \ref{th0m}.

\begin{proof}[Proof of Theorem \ref{th0m}]
  Consider $u(x,t)$ a
  solution of equation \eqref{equ} blowing up at
  time $T>0$.
 Assume that the origin is a
 non isolated blow-up point with $m(0)=m\ge 4$ is even.
%
%
%
%
%
We proceed in 4 steps:\\
- First, we recall from Vel\'azquez \cite{Vtams93}  the first order
Taylor expansion of $w_0(y,s)$ and show
that two coefficients are zero.\\
- Second, we introduce some geometrical transformation as a crucial tool in the proof.\\
- Third, we give the next order in the Taylor expansion
and
show that one coefficient
is zero. \\
- Forth, we take $m=4$ and justify \eqref{discriminant}.

\bigskip

\textbf{Step 1:
First order 
Taylor expansion}

 As stated in \eqref{Am}, we know from Vel\'azquez \cite{Vtams93} that
\begin{equation}\label{A4m}
  w_0(y,s)= 1+ e^{(1-\frac m2)s}
\sum_{j=0}^mC_{m,j}h_{m-j}(y_1)h_{j}(y_2)+o(e^{(1-\frac m2)s})
\end{equation}
for some real coefficients $C_{m,j}$ for $j=0,\dots,m$ such that the multilinear form 
\begin{equation}\label{defmult-m}
\sum_{j=0}^m C_{m,j}y_1^{m-j}y_2^j
\end{equation}
is non zero and nonpositive.
From the alternative due to Vel\'azquez \cite{Vtams93} and given on
page \pageref{resvel}, we know that this multiform has (at least) one
direction of degeneracy.
Up to making a rotation of coordinates, we may assume that the degeneracy direction is the axis $\{y_1=0\}$, which means that 
\begin{equation}\label{c44m}
C_{m,m}=0.
\end{equation}
Moreover, in order to guarantee the nonpositivity of the multilinear form, we directly see that
\begin{equation}\label{c43m}
C_{m,m-1}=0,
\end{equation}
otherwise, for $|y_2|$ large enough, the multilinear form may achieve positive values. 

\bigskip

\textbf{Step 2: A geometrical transformation}

As already written in the introduction, 
 a crucial idea lays at the heart of our strategy. It consists in remarking that any Taylor expansion of $w_0(y,s)$ can be translated into a Taylor expansion of $w_b(z,s)$, for any other point $b\in \m R^2$ (not necessarily a blow-up point), thanks to the following relation which follows from the similarity variables definition \eqref{eqw}:
\begin{equation}\label{wbw000}
w_b(y_b,s)=w_0(y,s)\mbox{ where }y=be^{\frac s2}+y_b.
\end{equation}
With a suitable choice of $b$, the uniform boundedness of $w_b$ in
$L^\infty$ induces some cancelations of the coefficients appearing in
the Taylor expansion of $w_0$.

\bigskip

\textbf{Step 3: Second order Taylor expansion}

Following the first order expansion \eqref{A4m} together with
\eqref{c44m} and \eqref{c43m}, we may use Proposition \ref{propexp} to find the next order in the Taylor expansion:
\begin{align}
  w_0(y,s)=& 1+ e^{(1-\frac m2)s}\sum_{j=0}^{m-2}C_{m,j}h_{m-j}(y_1)h_{j}(y_2) \label{expw0m}\\
           &+e^{\frac
             {1-m}2s}\sum_{j=0}^{m+1}C_{m+1,j}h_{m+1-j}(y_1)h_{j}(y_2)
             +\bar v_0(y,s) \nonumber
\end{align}
with
\begin{equation}\label{interm0}
\|\bar v_0(s)\|_{L^q_\rho}=O(se^{-\frac m2s})\mbox{ as }s\to \infty,
\end{equation}
for any $q\ge 2$,
for some real constants $C_{m+1,j}$ for $j=0,\dots,m+1$. We claim that
$C_{m+1,m+1}= 0$. Indeed, 
%
%
%
let us assume by contradiction that $C_{m+1,m+1}\neq 0$ and find a contradiction. Following Step 2 and choosing 
\begin{equation}\label{defb0}
b=(0,Ae^{\frac{s_0}2}),
\end{equation}
where $A$ and $s_0$ will be taken large enough (in modulus for $A$), we show the following (note that $b$ may or may not be a blow-up point):
\begin{lem}\label{lemcrit}For any $A\in \m R$ such that $|A|\ge 1$,
it holds that
\[
w_{b,0,0}(s_0)-1\sim C_{m+1,m+1} A^{m+1} e^{\frac {(1-m)}2 s_0}\mbox{ as }s_0\to \infty,
\]
where $b$ is given by \eqref{defb0}, and $w_{b,0,0}(s_0)$ is the coordinate of $w_b(y,s_0)$ along the polynomial $h_0(y_1)h_0(y_2)\equiv 1$ as in \eqref{defvij}.
\end{lem}


\noindent Let us first use this lemma to find a contradiction, then
prove the lemma.\\
Choosing $A$ of the sign of $C_{m+1,m+1}$,
then taking $s_0$ large enough, we see that
\begin{equation}\label{hided}
w_{b,0,0}(s_0)>1.
\end{equation}
In other words, $w_b$ satisfies the following blow-up criterion we proved in \cite{MZcpam98}:
\begin{prop}[Blow-up criterion for nonnegative solutions of
(\ref{eqw})]\label{problo}
Let $w$ be a nonnegative solution of (\ref{eqw}) which
satisfies 
\begin{equation*}
w_{0,0}(s_1)>1
\end{equation*}
for some $s_1\in\m R$. Then, $w$ cannot be defined for all $(y,s)\in \m R^2 \times [s_1,\infty)$. 
\end{prop}
\begin{nb}
  This proposition is valid only for nonnegative solutions. For
  unsigned solutions, 
  it should be replaced by a twin statement given in Proposition 1.2
  page 111 in \cite{MZma00}.
  \end{nb}
\begin{proof}
See Proposition 3.8 page 164 in \cite{MZcpam98}.
\end{proof}
Since $w_b$ exists by definition \eqref{defw} for all $(y,s)\in \m
R^2\times [-\log T, +\infty)$,  a contradiction follows from
\eqref{hided} and this proposition. Thus, $C_{m+1,m+1}=0$, provided that we prove Lemma \ref{lemcrit}.

\medskip

\begin{proof}[Proof of Lemma \ref{lemcrit}]
%
 %
Using the transformation \eqref{wbw000} together with the definition \eqref{defb0} of $b$, we see that
\begin{equation}\label{vbv0}
  v_b(y_b,s_0)=v_0(y,s_0)\mbox{ with }
y_1=y_{b,1},\;\;y_2=y_{b,2}+A, 
\end{equation}
hence, we write from \eqref{expw0m}
\begin{align}
  v_b(y_b,s_0)&= e^{(1-\frac m2)s_0}\sum_{j=0}^{m-2}C_{m,j}h_{m-j}(y_{b,1})h_{j}(y_{b,2}+A) \\
           &+e^{\frac
             {1-m}2s_0}\sum_{j=0}^{m+1}C_{m+1,j}h_{m+1-j}(y_{b,1})h_{j}(y_{b,2}+A)
             +\bar v_b(y_b,s_0) \label{vb}
%
\end{align}
with
\begin{equation}\label{defvbb}
\bar v_b(y_b,s_0)=\bar v_0(y,s_0).
\end{equation}
Recalling the classical relation
\[
h_j'=j h_{j-1}
\]
related to the Hermite polynomials \eqref{defhj}, we may use a Taylor expansion to derive the following binomial relation
\begin{equation}\label{binomial}
h_j(\xi+A)=\sum_{m=0}^j  \binom jm A^j h_{j-m}(\xi).
\end{equation}
Using this relation and projecting the expression \eqref{vb}, we derive the following expression
\[
v_{b,0,0}(s_0)= A^{m+1} C_{m+1,m+1}e^{\frac {1-m}2 s_0}+\bar v_{b,0,0}(s_0)
\]
where we have used the classical orthogonality relation
\begin{equation}\label{ortho}
\int_{\m R} h_l(\xi)h_j(\xi) \rho(\xi) d\xi = 2^jj! \delta_{l,j}. 
\end{equation}
 We need to estimate $\int \bar v_b(y_b,s_0)\rho(y_b) dy_b$ in order to conclude. 
By definition, we write
\begin{align*}
\bar v_{b,0,0}(s_0)
&=\int \bar v_0(y_1,y_2,s_0) \rho(y-(0,A))dy.
\end{align*}
Now, let us compute
\[
\rho(y-(0,A))=\rho(y) \exp\left(\frac A2 y_2\right)\exp\left(-\frac{A^2}4\right).
\]
Since $\int \exp(Ay_2)\rho(y) dy\le Ce^{A^2}$, 
using the Cauchy-Schwartz inequality together with \eqref{interm0}, we write
\begin{equation}\label{vb00}
|\bar v_{b,0,0}(s_0)|\le
 Ce^{\frac { A^2}4}\left(\int \bar v_0(y,s_0)^2\rho(y) dy \right)^{1/2}\le
 Ce^{\frac {A^2}4}s_0e^{-\frac m2 s_0}.
\end{equation}
This concludes the proof of Lemma \ref{lemcrit}. 
\end{proof}This also concludes the proof of the fact that $C_{m+1,m+1}=0$. From
\eqref{expw0m}, se see that the Taylor expansion
\eqref{do}
holds.

\bigskip

\textbf{Step 4: Proof of \eqref{discriminant} when $m=4$}

Using \eqref{do},
we see that the 4-form introduced in
  \eqref{defmult-m} reads as follows:
  \begin{equation*}
    C_{4,0}x_1^4+C_{4,1}x_1^3x_2+C_{4,2}x_1^2x_2^2.
  \end{equation*}
  Since this form is non-zero and nonpositive, dividing by $x_1^2$, we see that 
  \eqref{discriminant} follows.
  
\medskip

This concludes  the proof of Theorem \ref{th0m}.
\end{proof}
\section{Rigidity
in the geometry of the blow-up set
  when $m=4$}\label{sec0}
 This section is devoted to the proof of Theorem \ref{cor0m} when $m=4$.
 The proof when $m\ge 6$ is different and will be given in Section \ref{sec0m}.


\medskip

Consider $u(x,t)$ a
solution of equation \eqref{equ} blowing up at
  time $T>0$.
 Assume that the origin is a
 non isolated blow-up point with $m(0)=4$. Consider $a_n$ a sequence
 of non-zero blow-up points $a_n$ converging to the origin.
 
\medskip

We proceed in 4 subsections to give the proof:\\
- In Subsection \ref{P2}, we prove item (i) and the first part of item
(iii). In particular, 
after a change of variables, we identify 3 possible
regimes for the convergence of
$a_n$:
superquadratic, quadratic and subquadratic.\\
- In Subsection \ref{P3}, we show that the only subquadratic regime follows a
$3/2$ law, and determine its constant.\\
- In
Subsection \ref{P4}, we focus on the quadratic regime, and show that its
``constant'' enjoys at most 3 possible values. This will give the
second part of item (iii), which clearly implies item (ii).\\
- Finally, in Subsection \ref{secunif}, we give the proof of some
propositions which were used in the previous subsections, and which
are devoted to some refinements of the behavior of the solution.

\subsection{Superlinear behavior for neighboring points}\label{P2}

Recalling the alternative of Vel\'azquez \cite{Vtams93} given on
page \pageref{resvel}, we see that locally near the origin, the blow-up set is located
along some degeneracy directions of the multilinear form
\eqref{defmult}. From the choice of the axes we made in
Theorem \ref{th0m} and \eqref{do},
we know that the axis $\{y_1=0\}$ is a degeneracy
direction. Moreover, using
  \eqref{discriminant}, we remark that if
  \[
C_{4,1}^2-4C_{4,0}C_{4,2}=0\mbox{ and }C_{4,2}\neq 0,
\]
then, we may have another (different) degeneracy direction for the multilinear form
\eqref{defmult}, namely the line of equation
$C_{4,1}x_1+2C_{4,2}x_2=0$, no more. Thus, item (i) holds for the
sequence $a_n$. Since these
two straight lines are orthogonal (respectively) to $(0,1)$ and
$(C_{4,1}, 2C_{4,2})$, the first part of item (iii) follows too.

\medskip

Up to a rotation and symmetry
of the axes, together with the extraction of a subsequence still denoted the
same, we may assume that $a_n$ converges along the $x_2$-axis,
(which means that
\begin{equation}\label{one}
  \aa = o(\bb)
\end{equation}
as $n\to \infty$), and that
\begin{equation}\label{pos}
\forall n\in \m N,\;\aa\ge 0\mbox{ and }\bb \ge 0.
\end{equation}
Note that after such a change of variables, the Taylor
expansion
\eqref{do}
remains valid, with
possibly different
coefficients, denoted $C_{4,j}$ for $j=0,1,2$ and $C_{5,j}$ for $j=1,\dots,4$.

\medskip

Up to further extracting a subsequence still denoted the same, we are in
one of the following cases as $n\to \infty$:\label{3cases}\\
\textbf{- (subquadratic regime)} with $\aa \gg \bb^2$,\\
\textbf{- (quadratic regime)} with $\aa \sim L\bb^2$ for some $L>0$,\\
\textbf{- (superquadratic regime)} with $\aa \ll \bb^2$.

\medskip

In the two following
subsections,
we investigate the subquadratic then
the quadratic regimes for the sequence $a_n$, in order to conclude the
proof of Theorem \ref{cor0m}.

\bigskip

\subsection{
  Subquadratic regimes for neighboring blow-up points}\label{P3}

In this
subsection,
we assume that
\begin{equation}\label{noquad}
a_{n,1}\gg \bb^2\mbox{ as }n\to \infty.
\end{equation}
We will show that the sequence $a_n$ necessarily
follows a $3/2$ power law. We will also determine the ``constant'' in front
of such a law.

\medskip

Since $a_n\neq 0$, we see from \eqref{one} and \eqref{noquad} that
\begin{equation}\label{an2>0}
  \bb >0\mbox{ and }\aa>0.
\end{equation}
We proceed in 4
parts
to prove the $3/2$ law for $a_n$:\\
- In Part 1,
we prove that $C_{4,2}=C_{4,1}=0$ and $C_{4,0}<0$.\\
- In Part 2.
we prove that $C_{5,4}=0$ and $a_{n,1}=O\left(\bb^{\frac 32}\right)$ as $n\to \infty$.\\
- In Part 3,
assuming that $\frac {\aa}{\bb^{\frac 32}}$ has a non-zero
limit, we determine that limit in terms of the Taylor coefficients of
the solution.\\
- In Part 4,
we rule out the case where  $\aa= o(\bb^{\frac 32})$.

\bigskip

\textbf{Part 1:
  Proof of the fact  that $C_{4,2}=C_{4,1}=0$ and $C_{4,0}<0$}

We claim that it is enough to prove that
\begin{equation}\label{c42-}
C_{4,2}=0.
\end{equation}
Indeed, if this holds, recalling that $C_{4,4}=C_{4,3}=0$ in the
multilinear form \eqref{defmult}, thanks to \eqref{do}, we get the
following simpler expression for the form:
\[
C_{4,0}y_1^4+C_{4,1}y_1^3y_2.
\]
Since the is nonpositive and non-zero as stated right after \eqref{defmult}, we necessarily have
\[
C_{4,1}=0
\] 
and
\begin{equation}\label{c40<0-}
C_{4,0}<0.
\end{equation}
Let us then focus on the proof of \eqref{c42-}. Proceeding by contradiction, we assume that
\begin{equation}\label{assume}
C_{4,2}\neq 0. 
\end{equation}
Let us reach a contradiction in this case.

\medskip

We proceed as we did in
Step 3 in the Proof of Theorem \ref{th0m} 
to show that
$C_{m+1,m+1}=0$:
we
will derive information related to $w_b$ defined by \eqref{wbw000} in
Step 2 of the Proof of Theorem \ref{th0m}.
However, we will not use the expression \eqref{defb0} for $b$. We will instead take
\begin{equation}\label{defb1-}
b=a_n.
\end{equation}
There is also another difference with the proof of the fact that $C_{m+1,m+1}=0$, in the sense that the contradiction will follow from the behavior of $v_{b,1,0}$ instead of $v_{b,0,0}$. 
From \eqref{an2>0} and \eqref{one}, we 
introduce $B_n$ such that
\begin{equation}\label{defbn-}
a_{n,1} =B_n a_{n,2}\mbox{ with }B_n = o(1)\mbox{ as }n\to \infty.
\end{equation}
Following \eqref{an2>0} and the fact that $a_n \to 0$, given any $A>0$
(to be take large enough later), we also introduce
\begin{equation}\label{defsan}
  \san\to +\infty\mbox{ as }n\to \infty
\end{equation}
  such that
\begin{equation}\label{an2}
a_{n,2} = Ae^{-\frac \san 2},\mbox{ hence }a_{n,1} = B_n a_{n,2} = A B_ne^{-\frac \san 2}. 
\end{equation}
In particular, the transformation \eqref{wbw000} reads here as 
\begin{equation}\label{wbw-}
w_b(y_b,s)=w_0(y,s)\mbox{ with }
y_1 = y_{b,1}+AB_ne^{\frac\tau 2},\;\;y_2 = y_{b,2}+Ae^{\frac \tau 2}\mbox{ and }\tau = s-\san.
\end{equation}
Using
\eqref{do},
we write a new version of \eqref{vb} adapted to our new choice of $b$ in \eqref{defb1-}:
\begin{align}
v_b(y_b,\san)&=e^{-\san}\sum_{j=0}^2C_{4,j}h_{4-j}(y_{b,1}+AB_n)h_{j}(y_{b,2}+A)\nonumber\\
&+e^{-\frac 32\san}\sum_{j=0}^4C_{5,j}h_{5-j}(y_{b,1}+AB_n)h_{j}(y_{b,2}+A)+\bar v_b(y_b,\san)\label{vb1--}
\end{align}
where $\bar v_b(y_b,\san)$ is defined in \eqref{defvbb} (with $s_0$
replaced by $\san$).
Using the binomial relation \eqref{binomial} together with the
orthogonality relation given in \eqref{ortho}, we may derive from
\eqref{vb1--} a long formula for $v_b(y_b,\san)$. In fact, we won't do
that, since our argument uses only the projections of $v_b(y_b,\san)$
on $h_0h_0$, $h_1h_0$ and $h_0h_1$, with the notation
\[
h_lh_j\equiv h_l(y_{b,1})h_j(y_{b,2}),
\]
i.e. the projections on the expanding modes $\lambda=1$ and
$\lambda=\frac 12$ of the linear operator $\q L$ \eqref{defL}. For that reason, we will introduce a visual transcription of those 3 projections, in the form of a table whose columns are $v_{b,0,0}(\san)$,  $v_{b,1,0}(\san)$ and  $v_{b,0,1}(\san)$, and whose lines bear the name of the coeffients $C_{i,j}$ present in \eqref{vb1--}, together with the rest term $\bar v_b(y_b,\san)$. More precisely, this is the table:

\bigskip

\begin{center}
\begin{tabular}{|c|c|c|c|}
\hline
&$v_{b,0,0}(\san)$  &$v_{b,1,0}(\san)$ &$v_{b,0,1}(\san)$\\
\hline
&&&\\
$C_{4,i}$&$C_{4,i} A^4 B_n^{4-i}e^{- \san}$&$(4-i)C_{4,i}A^3 B_n^{3-i}e^{-\san}$&$iC_{4,i} A^3 B_n^{4-i}e^{-\san}$\\
\hline
&&&\\
$C_{5,i}$&$C_{5,i}A^5 B_n^{5-i}e^{-\frac 32 \san}$&$(5-i)C_{5,i} A^4 B_n^{4-i}e^{-\frac 32 \san}$&$iC_{5,i} A^4 B_n^{5-i}e^{-\frac 32 \san}$\\
\hline
&&&\\
$\bar v_b(y_b,\san)$ &$\bar v_{b,0,0}(\san)$&$\bar v_{b,0,1}(\san)$  &$\bar v_{b,1,0}(\san)$\\ 
\hline
\end{tabular}
\end{center}

\bigskip

\noindent Arguing as with \eqref{vb00}, we see that
\begin{equation}\label{exp-bvb10}
|\bar v_{b,i,j}(\san)|\le C(A) \san e^{-2 \san}\mbox{ with }
(i,j)=(0,0),\;\;(1,0)\mbox{ or }(0,1).
\end{equation}
Using this, together with the table and the smallness of $B_n$ written in \eqref{defbn-}, we see that
\begin{align}
v_{b,1,0}(\san) =& 2A^3B_ne^{-\san}(C_{4,2}+O(B_n))
+A^4e^{-\frac 32 \san}(C_{5,4}+O(B_n))\label{vb10-}\\
&+O(C(A) \san e^{-2 \san})\nonumber
\end{align}
as $n\to \infty$.
In order to know which term is dominant between the two appearing in
the expansion of $v_{b,1,0}(\san)$, 
let us introduce the ratio
\begin{equation}\label{defrn-}
R_n = \frac{A^3B_ne^{-\san}}{A^4e^{-\frac 32 \san}}=\frac{B_n}{Ae^{-\frac\san 2}}= \frac{a_{n,1}}{a_{n,2}^2}\to \infty
\end{equation}
from \eqref{defbn-}, \eqref{an2} and \eqref{noquad}.
In particular, it holds that
\begin{equation}\label{maher}
  B_n \gg A e^{-\frac \san 2}\mbox{ as }n\to \infty.
\end{equation}
Since $C_{4,2}\neq 0$ by the contradiction hypothesis \eqref{assume}, using \eqref{vb10-}, we write
\begin{equation}\label{v1b0-}
v_{b,1,0}(\san) \sim 2A^3B_nC_{4,2}e^{-\san}\mbox{ as }n\to \infty.
\end{equation}
Let us explain now how we reach a contradiction, first formally, then in a
rigorous way.

\medskip

Proceeding formally first, we restart the argument from \eqref{wbw-} with $\tau
=s-\san\ge 0$ this time. All the estimates run smoothly and we
end-up with the following modification of the table (note that we evaluate functions at $s=\san+\tau$, hence exponential factors appear in the estimates, complying with the expanding nature of the considered projections): 

\medskip

\begin{center}
\begin{tabular}{|c|c|c|c|}
\hline
&$v_{b,0,0}(\san+\tau)$  &$v_{b,1,0}(\san+\tau)$ &$v_{b,0,1}(\san+\tau)$\\
\hline
&&&\\
$C_{4,i}$&$C_{4,i}e^\tau A^4 B_n^{4-i}e^{- \san}$&$(4-i)C_{4,i}e^{\frac \tau 2} A^3 B_n^{3-i}e^{-\san}$&$iC_{4,i}e^{\frac\tau 2} A^3 B_n^{4-i}e^{-\san}$\\
\hline
&&&\\
$C_{5,i}$&$C_{5,i}e^\tau A^5 B_n^{5-i}e^{-\frac 32 \san}$&$(5-i)C_{5,i}e^{\frac \tau 2} A^4 B_n^{4-i}e^{-\frac 32 \san}$&$iC_{5,i}e^{\frac \tau 2} A^4 B_n^{5-i}e^{-\frac 32 \san}$\\
\hline
&&&\\
$\bar v_b(y_b,\san+\tau)$ &$\bar v_{b,0,0}(\san+\tau)$&$\bar v_{b,0,1}(\san+\tau)$  &$\bar v_{b,1,0}(\san+\tau)$\\ 
\hline
\end{tabular}
\end{center}

\bigskip

\noindent Selecting the column on $v_{b,1,0}(\san+\tau)$, we write 
\begin{equation}\label{v1b0t-}
v_{b,1,0}(\san+\tau) \sim 2e^{\frac \tau 2}A^3B_nC_{4,2}e^{-\san}\mbox{ as }n\to \infty,
\end{equation}
which shows an increase in the size of $v_{b,1,0}(\san+\tau)$ with
$\tau$. Taking an arbitrarily small $\delta_0>0$ then choosing $\tau = \tau_n$ such that
\[
2e^{\frac {\tau_n} 2}|A^3B_nC_{4,2}|e^{-\san}=\delta_0,
\]
which means that
\begin{equation}\label{deftau-}
\tau_n = 2 \san+2 \log \frac {\delta_0}{2|A^3B_nC_{4,2}|}\to \infty
\end{equation}
(use \eqref{defsan} and \eqref{defbn-}),
we see that
\begin{equation}\label{hadaf}
|v_{b,1,0}(\san+\tau_n)| \sim \delta_0,
\end{equation}
and this violates the following convergence, uniform with the respect
to the blow-up point:
\begin{prop}[Uniform convergence of $v_a(y,s)$ to $0$ when $a$ is a
  blow-up point]
  It holds that
\begin{equation}\label{uniliou}
\sup_{a\in \q S}\|v_a(s)\|_{L^2_\rho}\to 0\mbox{ as }s\to \infty.
\end{equation}
In particular, 
for any $i\in \m N$ and $j=0,\dots,i$ with $(i,j)\neq(0,0)$, we have 
\begin{equation}\label{uij}
\sup_{a\in \q S}|v_{a,i,j}(s)|\to 0\mbox{ as }s\to \infty.
\end{equation}
\end{prop}
\begin{proof}
  Clearly, \eqref{uij} directly follows from \eqref{uniliou}, thanks to the $L^2_\rho$
  projection defined in \eqref{defvij}. As for \eqref{uniliou}, it is
  a direct consequence of the Liouville theorem of \cite{MZcpam98} and \cite{MZma00}. For a proof, one may adapt with no difficulty the proof of
  Proposition 2.2 page 11 in Khenissy, Rebai and Zaag
  \cite{KRZihp11}. 
\end{proof}
Clearly, a contradiction follows from \eqref{hadaf} and \eqref{uij} (note that $\san+\tau_n\to
\infty$ by \eqref{defsan} and \eqref{deftau-}), justifying that $C_{4,2}=0$.
Unfortunately, the change of variables starting from \eqref{wbw-} can
be justified only when $\tau$ stays bound, and this is not the case in
\eqref{deftau-}.

\medskip

Fortunately, there is a rigorous way to justify \eqref{v1b0t-}. One has
simply to take \eqref{v1b0-} as initial data at $s=\san$, then integrate  for $s\ge \san$ the
following ODE satisfied by $v_{b,1,0}(s)$
which we recall from \eqref{eqvij}:
\begin{equation}\label{eqvb10-}
v_{b,1,0}'(s) = \frac 12 v_{b,1,0}(s) + \int v_b(y,s)^2 \k_1(y_1)
\rho(y) dy,
  \end{equation}
where $k_1(y_1) = \frac {y_1}2$ was already introduced in \eqref{defk}. 

\medskip

In order to do so, we need to evaluate the size of $v_b(y,s)$, and this is possible if
we evaluate it at $s=\san$, then integrate the PDE satisfied by $v_b(y,s)$ for $s\ge
s_n$. Let us start then by
evaluating the size of  $v_b(y,s)$ at $s=\san$.\\
Before that, recalling that
\eqref{do}
holds in $L^4_\rho$, we may use
the trick based on the Cauchy-Schwarz inequality we used for
\eqref{vb00} to prove that $\bar v_b(\san)$ defined in \eqref{defvbb}
and appearing in \eqref{vb1--} satisfies
\begin{equation}\label{vbar2}
\|\bar v_b(\san)\|_{L^2_\rho} \le C(A)\san e^{-2\san}.
\end{equation}
Hence, we may use \eqref{vb1--} together with the
binomial relation \eqref{binomial} in order to
show that 
\begin{equation}\label{fai}
\|v_b(y_b,\san) \|_{L^2_\rho} \le CA^2e^{-\san},
\end{equation}
for $n$ large enough.
Then, integrating the PDE \eqref{eqv}, we show that for $s\ge \san$, the solution will stay as small as its value at $s=\san$. More precisely, we claim the following:
\begin{prop}[Uniform bound for $v_{a_n}$]\label{propunif}
 There exists a universal constant $\cc_0>0$ such that for all $A\ge
 1$ and $\dd\ge 1$, for $n$ large enough, we have 
\[
\forall s\in [\san, \dd \san],
  \;\;\|v_{a_n}(s)\|_{L^2_\rho}\le \cc_0 A^2e^{-\san}.
\]
\end{prop}
\begin{proof}
  See subsection \ref{secunif} below.
\end{proof} 
\begin{nb}
  The result can not hold if one replaces $v_{a_n}$ by $v_{x_0}$ where $x_0$
  is not a blow-up point. Indeed, from Giga and Kohn \cite{GKcpam89},
  we know in that case that $w_{x_0}(s)\to 0$, hence $v_{x_0}(s)\to -1$ as $s\to \infty$.
  \end{nb}
\begin{nb}
In fact, our next argument needs the same estimate with better norms,
namely the $L^4_\rho$. In order to justify that, we give two
facts:\\
(i) First, given an arbitrary $\tau_0>0$, note that estimate \eqref{fai} actually holds for all $s\in [\san-\tau_0,
\san]$, by the same argument (based on the table given right before
\eqref{v1b0t-}, which holds uniformly for $\tau\in [-\tau_0,0]$ with
the same proof). This will provide the $L^2_\rho$ bound of Proposition
\ref{propunif} for any
$s\in [\san, \dd \san]$.\\
(ii) Second,
using Lemma \ref{lemVel},
note that the $L^4_\rho$
estimate follows from the $L^2_\rho$ estimate, after some time shift
$\tau^*=\tau^*(4,2)$ (by the way, using the nonlinear equation
\eqref{eqv} satisfied by $v_{a_n}$ and the $L^\infty$ bound
\eqref{b1}, we see that $\partial_s|v|\le (\q L +C)|v|$ is the
distribution sense, hence Lemma \ref{lemVel} applies).  
Choosing the constant $\tau_0$ of item (i) equal
to $\tau^*$ gives the $L^4_\rho$ estimate for any $s\ge \san$.
\end{nb}
Note first from the definition \eqref{deftau-} of $\tau_n$ and
\eqref{maher} that there exists $\dd \ge 1$ such that $\san+\tau_n\in [\san,\dd\san]$.
Plugging the estimate of Proposition \ref{propunif}  in the quadratic term of \eqref{eqvb10-} (we
need in fact the $L^4_\rho$  estimate, which holds thanks to the
remark following Proposition \ref{propunif}), then integrating for
$s'\in [\san, s]$ (remember that $|v_{a_n,1,0}(s')|$ is uniformly
bounded from \eqref{b1}), we obtain (remember that $b=a_n$ from \eqref{defb1-}):
\begin{equation}\label{intvb10}
\forall s\in [\san, \dd \san],
  \;\;|v_{b,1,0}(s)-e^{\frac \tau 2}v_{b,1,0}(\san)|\le Ce^{\frac \tau 2}A^4 e^{-2\san},\mbox{ where }\tau=s-\san.
\end{equation}
Using \eqref{v1b0-} and \eqref{maher}, we see that \eqref{v1b0t-} is
justified. Taking $\tau=\tau_n$ defined in \eqref{deftau-}, we see
that \eqref{hadaf} holds, which contradicts Proposition
\ref{propunif}. Therefore, \eqref{assume} doesn't hold. Thus,
$C_{4,2}=0$ and also $C_{4,1}=0$ and $C_{4,0}<0$ by the argument
presented at the beginning of this
part.

\bigskip

\textbf{Part 2:
  Proof of the fact that $a_{n,1}=O\left(a_{n,2}^{\frac
      32}\right)$ as $n\to \infty$}

This estimate will be achieved through 4 lemmas below, numbered
\ref{cl2} to \ref{cl5}, where we successively prove that  $a_{n,1}=O\left(a_{n,2}^{\frac
      43}\right)$, then  $a_{n,1}=o\left(a_{n,2}^{\frac
      43}\right)$, then $C_{5,4}=0$ and finally that $a_{n,1}=O\left(a_{n,2}^{\frac
      32}\right)$ as $n\to \infty$. Let us state the first lemma:

\begin{lem}\label{cl2}It holds that $a_{n,1}=O\left(a_{n,2}^{\frac
      43}\right)$ as $n\to \infty$.
\end{lem}
\begin{proof}
  Since $C_{4,2}=C_{4,1}=0$ by
  Part 1,
  arguing as in that part, we can
get an improved version of \eqref{vb10-}, thanks to the table given right
before
that estimate, the smallness of $B_n$ written in
\eqref{defbn-} together with \eqref{exp-bvb10}:
\[
v_{b,1,0}(\san) =4A^3C_{4,0}B_n^3 e^{-\san}+ A^4e^{-\frac 32 \san}(C_{5,4}+O(B_n))+O(\san e^{-2 \san})
\]
as $n\to \infty$. Here, a discussion arises, according to the ratio of the two terms appearing in this expression:
\[
\frac{A^3B_n^3 e^{-\san}}{A^4 e^{-\frac 32 \san}}=
\frac{B_n^3}{Ae^{-\frac\san 2}}=
\frac{a_{n,1}^3}{a_{n,2}^4},
\]
by definitions \eqref{defbn-} and \eqref{an2} of $B_n$ and
$\san$. Arguing by contradiction to prove
Lemma \ref{cl2}
and recalling that $C_{4,0}\neq 0$ by \eqref{c40<0-}, we see that we have
\begin{equation}\label{haspot}
B_n^3 \gg A e^{-\frac \san 2}\mbox{ and }
v_{b,1,0}(\san) \sim 4A^3C_{4,0}B_n^3 e^{-\san}\mbox{ as }n\to \infty,
\end{equation}
at least for a subsequence denoted the same. 
Arguing as we did for \eqref{v1b0t-} and using this time the table
right before
that estimate, we may show that 
\begin{equation}\label{vb10t-}
v_{b,1,0}(\san+\tau) \sim 4e^{\frac \tau 2} A^3B_n^3 C_{4,0}e^{- \san}\mbox{ as }n\to \infty, 
\end{equation}
then, derive a contradiction as in \eqref{hadaf}, with
\[
\tau_n = -2 \log\frac{4|C_{4,0}|A^3B_n^3e^{-\san}}{\delta_0}
\]
which satisfies $\tau_n \to \infty$ from \eqref{defbn-},  hence $\san
+\tau_n \to \infty$ from \eqref{defsan}
and also $\san+\tau_n \le \dd\san$ for some $\dd \ge 1$, from \eqref{haspot}.
The question then reduces to justify \eqref{vb10t-}, for this choice
of $\tau_n$, and this can be done exactly as in
Part 1, 
relying on Proposition \ref{propunif}.
Thus, Lemma \ref{cl2} is proved.
\end{proof}

\bigskip

Now, we state the second lemma of
this part:
\begin{lem} \label{cl3}It holds that $a_{n,1}=o\left(a_{n,2}^{\frac 43}\right)$ as $n\to \infty$.
\end{lem}
\begin{proof}
Arguing by contradiction, we assume that up to some subsequence (still denoted the same), we have
\begin{equation}\label{cont}
\frac{a_{n,1}}{a_{n,2}^{\frac 43}}\to L\mbox{ as } n\to \infty,
\end{equation}
for some $L> 0$. By definitions \eqref{defbn-} and \eqref{an2} of $B_n$ and $\san$, we see that
\begin{equation}\label{bnsize}
B_n \sim LA^{1/3}e^{-\frac{\san}6}\mbox{ as } n\to \infty.
\end{equation}
Using \eqref{vb1--} as in
Part 1 and Lemma \ref{cl2},
we may use
\eqref{exp-bvb10} and
the table right
before
it
together with \eqref{bnsize} to derive the
 following expression for another coordinate, namely $v_{b,0,1}(\san)$ (there
 is a great advantage in considering this coordinate instead of its 
 twin  $v_{b,1,0}(\san)$, in the sense that with  $v_{b,0,1}(\san)$,
 there is no contribution of $C_{4,0}$) :
\begin{align*}
  v_{b,0,1}(\san)
&= A^4B_n e^{-\frac 32 \san}(4C_{5,4}+O(B_n))+O\left(C(A)\san e^{-2 \san}\right)\\
&  =4A^{4+\frac 13}L C_{5,4}e^{-\frac 53 \san}+O\left(e^{-\frac{11}6 \san}\right).
\end{align*}
By the same argument as in
Part 1 and Claim \ref{cl2}, 
we see that $C_{5,4}=0$ (the justification follows also in the same
way, thanks to Proposition \ref{propunif}).

\medskip

Repeating the same argument presented here in this step, we derive the following expression for the other coordinate (use the table right before \eqref{exp-bvb10}):
\begin{align*}
  v_{b,1,0}(\san)
&=e^{\frac \tau 2}\left\{ 4A^3C_{4,0}B_n^3 e^{-\san}+O(A^4 B_ne^{-\frac 32
    \san}) +O\left(C(A)\san e^{-2 \san}\right)
\right\}\\
& \sim 4e^{\frac \tau 2}A^4L^3 C_{4,0}e^{-\frac 32 \san}.
\end{align*}
Again, the growing factor implies that $C_{4,0}=0$, which is a contradiction by \eqref{c40<0-}. Thus, \eqref{cont} doesn't hold and
Lemma \ref{cl3} holds.
As for the rigorous proof,
it follows exactly as in
Part 1,
thanks to Proposition \ref{propunif}.
This concludes the proof of Lemma \ref{cl3}.
\end{proof}


Now, this is the third lemma of the
part:
\begin{lem} \label{cl4} It holds that $C_{5,4}=0$.
  \end{lem}
\begin{proof}
Proceeding by contradiction, we assume that
\begin{equation}\label{c54n0}
C_{5,4}\neq 0.
\end{equation}
This time, the argument follows from the behavior of
$v_{b,1,0}(s)$. Using
\eqref{maher}
and
Lemma \ref{cl3},
we see by definitions \eqref{defbn-} and \eqref{an2} of $B_n$ and $\san$ that
\begin{equation}\label{bnsize1}
Ae^{-\frac \san 2}\ll B_n \ll A^{\frac 13} e^{-\frac \san 6}\mbox{ as }n\to \infty.
\end{equation}
Proceeding as before, we may use the table right before
\eqref{v1b0t-}
to derive that  
\begin{equation*}
v_{b,1,0}(\san+\tau)\sim A^4  C_{5,4} e^{\frac \tau 2} e^{-\frac 32 \san}\mbox{ as }n\to \infty
\end{equation*}
(recall that in \eqref{vb1--}, $C_{4,1}=C_{4,2}=0$ from Part 1).
The growth factor $e^{\frac\tau 2}$ allows to get a contradiction as
in previous steps. Thus, \eqref{c54n0} doesn't hold and
$C_{5,4}=0$. The rigorous justification again follows exactly as in
Part 1,
thanks to Proposition \ref{propunif}. This concludes the proof
of Lemma \ref{cl4}.
\end{proof}

\bigskip

Now, this is the final lemma of the
part:
\begin{lem}\label{cl5} It holds that $\aa = O(\bb^{3/2})$ as $n \to \infty$.
  \end{lem}
\begin{proof}

  Since $C_{5,4}=0$ from
Lemma \ref{cl4},
  we go back again to the
table right before
\eqref{v1b0t-}
to derive the following expansion
\begin{equation}\label{foued}
  v_{b,1,0}(\san+\tau)=e^{\frac \tau 2} \left\{
4A^3 C_{4,0}B_n^3e^{-\san}+
    2A^4 B_n  e^{-\frac 32 \san}(C_{5,3}+O(B_n))+O\left(\san
      e^{-2\san}\right)\right\}
 \end{equation}
as $n\to \infty$. Making the ratio between the coefficients of the first two terms, we get
\begin{equation}\label{asma}
\frac{A^3 B_n^3 e^{-\san}}{A^4 B_n e^{-\frac 32
    \san}}=\frac{B_n^2}{Ae^{-\frac \san 2}}=\frac {\aa^2}{\bb^3},
\end{equation}
by definitions \eqref{defbn-} and \eqref{an2} of $B_n$ and
$\san$. Assuming by contradiction that
Lemma \ref{cl5}
doesn't hold for some subsequence still denoted the same, we see that
$B_n \gg \sqrt A e^{-\frac \san 4}$, hence,
\[
 v_{b,1,0}(\san+\tau) \sim 4e^{\frac \tau 2} A^3
 C_{4,0}B_n^3e^{-\san}\mbox{ as }n\to \infty,
\]
and the growth
factor shows that $C_{4,0}=0$, which is a contradiction from
\eqref{c40<0-}. The justification of this step follows as in
Part 1.
This concludes the proof of Lemma \ref{cl5}.
\end{proof}

\bigskip

\textbf{Part  3:
  Possible limits for $\frac {\aa}{\bb^{\frac 32}}$}

From Lemma \ref{cl5},
we may extract a subsequence (still denoted the same) such
that
\begin{equation}\label{deflim}
\frac {\aa}{\bb^{\frac 32}}\to L_0\mbox{ as }n\to \infty,
\end{equation}
for some $L_0\ge 0$. Let us assume that $L_0>0$ (note that the case
$L_0=0$ will be
ruled out in Part 4). Using \eqref{asma}, we see that
\[
B_n \sim L_0\sqrt A e^{-\frac \san 4}\mbox{ as }n\to \infty.
\]
From \eqref{foued}, we see that
\[
v_{b,1,0}(\san+\tau)\sim e^{\frac \tau 2} A^4B_n e^{-\frac 32\san}\left(
4C_{4,0}L_0^2+
    2C_{5,3}+O(\san e^{-\frac \san 2})\right)\mbox{ as }n\to \infty.
\]
As before, we see from the growth factor that
\[
4C_{4,0}L_0^2+
    2C_{5,3}=0.
\]
Since $C_{4,0}<0$ from \eqref{c40<0-} and $L_0>0$, we necessarily see
that $C_{5,3}>0$ and
\begin{equation}\label{L0}
L_0= \sqrt{-\frac{C_{5,3}}{2 C_{4,0}}}.
\end{equation}
The justification of this part can be done exactly as
Part 1.

\bigskip

\textbf{Part 4: Ruling out the case  where $\aa=
  o(\bb^{\frac 32})$ as $n\to \infty$}


Let us assume by contradiction 
that the limit defined in \eqref{deflim} is
zero. In other words,
\begin{equation}\label{L00}
L_0=0.
\end{equation}
From \eqref{maher} and \eqref{asma}, we see that 
\begin{equation}\label{imen}
Ae^{-\frac \san 2} \ll B_n \ll A^{\frac 12}e^{-\frac\san 4}\mbox{ as } n\to \infty.
\end{equation}
At the end of this part, 
we will reach a contradiction, ruling this case out.
In fact, our argument will follow from 3 steps:\\
- In Step 1, we show that  $a_{n,1}=O\left(a_{n,2}^2|\log
  a_{n,2}|\right)$ as $n\to \infty$.\\
- In Step 2, we prove that $C_{6,6}=C_{5,3} = C_{6,5}=0$.\\
- In Step 3, we reach a contradiction and finish the argument of Part 4.

\bigskip

\textbf{Step
1:
  Proof of the fact that  $a_{n,1}=O\left(a_{n,2}^2|\log
     a_{n,2}|\right)$ as $n\to \infty$}
 
Let us first show how the ratio between $a_{n,1}$ and $a_{n,2}^2|\log
a_{n,2}|$ naturally arises in our argument.
Using \eqref{foued} together with \eqref{imen}, we write:
\begin{equation}\label{vb10s}
  v_{b,1,0}(\san+\tau)=e^{\frac \tau 2} \left\{
    2A^4 B_n  e^{-\frac 32 \san}(C_{5,3}+O(B_n))+O\left(\san
      e^{-2\san}\right)\right\}
  \mbox{ as } n\to \infty.
\end{equation}
Clearly, the balance between the $C_{5,3}$ term and the error term is important. Considering the ratio between the two, we write
\begin{equation}\label{ratio}
\frac{A^4B_n e^{-\frac 32 \san}}{\san e^{-2 \san}}=A^5\frac{B_n}{\san Ae^{-\frac \san 2}}\sim \frac{A^5}2 \frac{a_{n,1}}{a_{n,2}^2|\log a_{n,2}|}\mbox{ as }n\to \infty
\end{equation}
thanks to the definitions \eqref{defbn-} and \eqref{an2} of $B_n$ and
$\san$. Thus, the ratio in the title of Step \SCinq $ $
appears.

\medskip

Proceeding by contradiction, we assume that (for a subsequence denoted the same), we have
\begin{equation}\label{contra0}
a_{n,1}\gg a_{n,2}^2|\log a_{n,2}|\mbox{ as }n\to \infty.
\end{equation}
Using \eqref{ratio} and
\eqref{imen},
we see that
\begin{equation}\label{bnsize2}
\san e^{-\frac \san 2}\ll B_n \ll e^{-\frac \san 4}\mbox{ as }n\to \infty.
\end{equation}
Using \eqref{vb10s}, we see that
\[
 v_{b,1,0}(\san+\tau) \sim    2 e^{\frac \tau 2} A^4 B_n  e^{-\frac 32 \san}(C_{5,3}+o(1))
  \mbox{ as } n\to \infty,
\]
and the growth factor implies that
%
%
\begin{equation}\label{clc53} 
C_{5,3}=0.
\end{equation}
The rigorous justification is exactly the same as in earlier steps,
relying on Proposition \ref{propunif}.

\bigskip


At this stage, we need to further refine the Taylor expansion of
$v_0(y,s)$ given in
\eqref{do},
taking into account that with respect to the expansion \eqref{vb1--},
we already have
$C_{4,1}=C_{4,2}
=C_{5,3}=C_{5,4}
=0$, from
Part 1, Lemma \ref{cl4} and \eqref{clc53}.
Using the strategy described in Proposition \ref{propexp}, we write:
\begin{align}
 v_0(y,s)=&-\frac {e^{-2s}}3C_{4,0}^2 \gamma_{4,4,0}h_0h_0
-\frac {e^{-2s}}2C_{4,0}^2 \gamma_{4,4,2}h_2h_0
+\left[C_{4,0}e^{-s}- e^{-2s}C_{4,0}^2 \gamma_{4,4,4}\right]h_4h_0\nonumber\\
&+e^{-\frac 32s}\sum_{j=0}^2C_{5,j}h_{5-j}h_j
+se^{-2s}C_{4,0}^2 \gamma_{4,4,6}h_6h_0
+e^{-2s}\sum_{j=0}^6C_{6,j}h_{6-j}h_j\nonumber\\
&+e^{-2s}C_{4,0}^2h_8h_0+O\left(se^{-\frac 52s}\right),\label{taylor2-}
\end{align}
for some real coefficients $C_{6,j}$ with $j=0,\dots,6$, 
where we have used the following notation
\[
\gamma_{l,m,n}= \int_{\m R} h_l(\xi)h_m(\xi) k_n(\xi)\rho(\xi) d\xi,
\]
with the polynomials defined in \eqref{defhj} and \eqref{defk}, the weight $\rho$ in \eqref{defro} and the convention
\[
\gamma_{l,m,n}=0\mbox{ if }l,\;m\mbox{ or }n\mbox{ is negative }, 
  \]
together with the notation
\[
h_lh_m\mbox{ which stands for }h_l(y_1) h_m(y_2).
\]
Following this new refined expansion,
we give first the following cancelation:
\begin{cl}\label{clazzabi} It holds that $C_{6,6}=C_{6,5}=0$.
\end{cl}
%
%
\begin{proof}
Using the expansion \eqref{taylor2-}, we give the following table, which
is an update of the former table given right  before  \eqref{exp-bvb10}:


\bigskip

\hspace{-3cm}
\begin{tabular}{|c|c|c|c|}
\hline
&$v_{b,0,0}(\san+\tau)$  &$v_{b,1,0}(\san+\tau)$
  &$v_{b,0,1}(\san+\tau)$\\
  \hline
  &&&\\
  $C_{4,0}^2 \gamma_{4,4,0}$&
 $-\frac{C_{4,0}^2}3 \gamma_{4,4,0}e^{-2\tau}e^{-2\san}$
                         &0&0\\
  \hline
  &&&\\
  $C_{4,0}^2 \gamma_{4,4,2}$
  & $-\frac{C_{4,0}^2}2\gamma_{4,4,2} e^{-\tau}A^2B_n^2e^{-2\san}$
    & $-C_{4,0}^2\gamma_{4,4,2} e^{-\frac 32\tau}AB_ne^{-2\san}$
   &0\\
\hline
&&&\\
$C_{4,0}$&$C_{4,0}e^\tau A^4 B_n^4e^{-\san}$&$4C_{4,0}e^{\frac \tau 2}
                                              A^3 B_n^3e^{-\san}$&0\\
  \hline
  &&&\\
  $C_{4,0}^2 \gamma_{4,4,4}$&
 $-C_{4,0}^2 \gamma_{4,4,4}  A^4 B_n^4e^{-2\san}$&$-4C_{4,0}^2 \gamma_{4,4,4} e^{-\frac \tau 2}A^3 B_n^3e^{-2\san}$&0\\
\hline
&&&\\
$C_{5,i}$&$C_{5,i}e^\tau A^5 B_n^{5-i}e^{-\frac 32
           \san}$&$(5-i)C_{5,i}e^{\frac \tau 2} A^4 B_n^{4-i}e^{-\frac
                   32 \san}$&$iC_{5,i}e^{\frac \tau 2} A^4
                              B_n^{5-i}e^{-\frac 32\san}$\\
\hline
&&&\\
$C_{4,0}^2\gamma_{4,4,6}$& $C_{4,0}^2\gamma_{4,4,6}e^\tau A^6 B_n^6 (\san+\tau)e^{-2\san}$ &$6C_{4,0}^2\gamma_{4,4,6}e^{\frac\tau 2}A^5 B_n^5 (\san+\tau)e^{-2\san}$ &0\\
\hline
&&&\\
$C_{6,i}$& $C_{6,i}e^\tau A^6 B_n^{6-i} e^{-2\san}$ &$(6-i)C_{6,i}e^{\frac \tau 2} A^5 B_n^{5-i} e^{-2\san}$ &$iC_{6,i}e^{\frac \tau 2} A^5 B_n^{6-i} e^{-2\san}$\\
\hline
&&&\\
$C_{4,0}^2$&$C_{4,0}^2e^{2\tau} A^8 B_n^8e^{-2\san}$ &$8C_{4,0}^2e^{\frac 32\tau} A^7 B_n^7e^{-2\san}$ &0\\
\hline
&&&\\
$O(se^{-\frac 52s})$&$e^\tau O((\san+\tau)e^{-\frac 52\san})$
&$e^{\frac \tau 2} O((\san+\tau)e^{-\frac 52\san})$
&$e^{\frac \tau 2} O((\san+\tau)e^{-\frac 52\san})$\\
\hline
\end{tabular}



\bigskip

Please keep in mind that this table is fully justified if $\tau$ is
zero or bounded. 
It happens that for our argument, we will need
$\tau=\tau_n$ to go to infinity as $n\to \infty$. In that case,
note that the terms bearing the coefficients $e^{2\tau}$ and
$e^{3\tau/2}$ are ``artificial'' and should not be taken into account
in our argument, including the formal one. This way, we will ignore
those terms in our formal argument. Fortunately, when it comes to the
rigorous argument, this table will be useful for $\tau=0$ or
bounded, 
and the conclusion will follow from the integration of the
PDE for $\tau$ ranging between $0$ and suitably large values.

\bigskip

\noindent-
\textit{Proof of the fact that $C_{6,6}=0$}.\\
Using \eqref{bnsize2} and the previous table, we may derive the following
estimate for large $\tau$:
\begin{equation}\label{formal}
v_{b,0,0}(\san +\tau) = e^\tau(A^6C_{6,6}+o(1))e^{-2\san} \mbox{ as
}n\to \infty,
\end{equation}
and the growth factor implies that
\begin{equation}\label{c66}
  C_{6,6}=0.
\end{equation}
As we have just written after the previous table, we don't take into
account the term with $e^{2\tau}$ is the formal derivation of
\eqref{formal}.
Now, as for the rigorous justification, it goes as in previous steps,
relying on the previous table only for bounded $\tau$, together with
Proposition \ref{propunif}. However, the rigorous proof is a little
more delicate than earlier,
since we need to take $A$ large enough in
order to conclude, as we explain in the following:\\
- first, relying on the previous table and \eqref{bnsize2}, we write the following
expansion for $\tau=0$ (note that it is different from \eqref{formal},
since we have lower order exponential terms in $\tau$ in the table
which cannot be neglected for small $\tau$):
\begin{equation}\label{initvb00}
  v_{b,0,0}(\san) =\left[-\frac{C_{4,0}^2}3
  \gamma_{4,4,0}+C_{6,6}A^6+o(1)\right]e^{-2\san}
  \mbox{ as }n\to \infty.
\end{equation}
- second, using Proposition \ref{propunif} (with $\dd =4$) and proceeding as for
\eqref{intvb10}, we integrate the ODE \eqref{eqvij} satisfied by
$v_{b,0,0}$ and write: 
\begin{equation}\label{intvb00}
\forall s\in [\san, \dd \san],
  \;\;|v_{b,0,0}(s)-e^\tau v_{b,0,0}(\san)|\le Ce^\tau A^4 e^{-2\san},\mbox{ where }\tau=s-\san.
\end{equation}
Using \eqref{initvb00}, and taking $A$ large enough so that $A^6$ term in
\eqref{initvb00} dominates its neighbor with $C_{4,0}^2$, and also the
$A^4$ term in the right-hand side of
\eqref{intvb00}, we see that \eqref{formal} is justified.



\bigskip

\noindent -
\textit{Proof of the fact that $C_{6,5}=0$}.\\
Proceeding as before, we may write the following expansion:
\begin{align}
v_{b,0,1}(\san +\tau) = e^{\frac \tau 2}&\left\{2A^4 B_n^3 e^{-\frac 32 \san}\left(C_{5,2}+O(B_n)\right)\right.\label{doremi}\\
&\left.+5A^5 B_n e^{-2\san}(C_{6,5}+O(B_n))+O\left(\san e^{-\frac 52 \san}\right)\right\}.\nonumber
\end{align}
Assuming by contradiction that $C_{6,5}\neq 0$, we see
from \eqref{bnsize2}
that
\begin{equation}\label{redosila}
v_{b,0,1}(\san +\tau) \sim 5e^{\frac \tau 2}A^5 B_n e^{-2\san}C_{6,5},
\end{equation}
and the growth factor leads to a contradiction as usual. Thus, $C_{6,5}=0$.
%
%
As for the rigorous justification, it is a little more complicated
than
usual.
Indeed, when $\tau =0$, we have $|v_{b,0,1}(\san)| \ll e^{-2\san}$, less than $ e^{-2\san}$, which is the bound Proposition \ref{propunif} gives on the quadratic term in the equation \eqref{eqvij} satisfied by $v_{b,0,1}$. In other words, Proposition \ref{propunif} is not enough to show that the linear term dominates the quadratic term when we integrate equation \eqref{eqvij}. Thus, we need to refine Proposition \ref{propunif} in the following:
\begin{prop}[Uniform expansion for $v_{a_n}$]\label{propexp4}
 Following Proposition \ref{propunif} and assuming that
  \begin{equation}\label{rough}
    v_0(y,s)= C_{4,0} e^{-s}h_4h_0+O(e^{-\frac 32s})\mbox{ as }s\to \infty
  \end{equation}
  (in $L^q_\rho$ for any $q\ge 2$), we claim 
   that for all $A\ge 1$ and $\dd \ge 1$,  for $n$ large enough, we
   have for all
$s\in[\san, \dd \san]$,
\[
  \|v_{a_n}(y,s)-C_{4,0}e^{-s}h_4(y_1)\|_{L^2_\rho}\le
C(A)e^{-\frac 32 \san}.
\]
\end{prop}
\begin{nb}As we wrote right after Proposition \ref{propunif}, here
  also, we have the same estimate valid in $L^4_\rho$, by the same
  parabolic regularity argument.
  \end{nb}
\begin{proof}
See Subsection \ref{secunif}.
\end{proof}
Let us use this proposition to justify \eqref{redosila} (note from
\eqref{taylor2-} that \eqref{rough} holds here).\\
Note first from \eqref{bnsize2} that the time $\tau_n$ which allows
\eqref{redosila} to be equal to some $\delta_0>$ satisfies $\tau_n \in [0,
\bar \dd \san]$, for some $\bar \dd >0$. Let us then apply Proposition
\ref{propexp4} with $\dd = \bar \dd +1$.\\
First, when $s=\san$, we may use \eqref{doremi} (which is fully
justified when $\tau=0$) and \eqref{bnsize2} to write (remember that
 $b=a_n$ from \eqref{defb1-}):
\begin{equation}\label{ini4}
v_{b,0,1}(\san)=5A^5 B_n e^{-2\san}C_{6,5}+O(B_n^3 e^{-\frac 32 \san})+O(\san e^{-\frac 52 \san}).
\end{equation}
Now, we recall the ODE \eqref{eqvij} satisfied by $v_{b,0,1}$:
\[
\forall s\ge \san,\;\;v_{b,0,1}'(s)=\frac 12 v_{b,0,1}(s)+\int v_b(y,s)^2 k_1(y_2)\rho(y)dy.
\]
Using Proposition \ref{propexp4} with the $L^4_\rho$ norm (use the
remark following the proposition), we may estimate the quadratic term as follows
 (note that the contribution of the term $[C_{4,0}e^{-s}h_4(y_1)]^2$ vanishes thanks to the orthogonality between $h_4(y_1)^2$ and $h_1(y_2)$ in $L^2_\rho$), and write the following differential inequality on $v_{b,0,1}(s)$:
\begin{align*}
\forall s\ge \san,\;\;|v_{b,0,1}'(s)-\frac 12 v_{b,0,1}(s)|
  \le & C(s-\san)e^{-s}A^4e^{-2\san}+C(A)e^{-s}e^{-\frac 32 \san}\\
  +&C(s-\san)^2A^8e^{-4\san}+C(A)^2e^{-3\san}.
\end{align*}
Integrating this equation, we see that
\begin{equation*}
\forall \tau \ge 0,\;\;|v_{b,0,1}(\tau+\san)-e^{\frac \tau 2}v_{b,0,1}(\san)|
\le C(A)e^{\frac \tau 2}e^{-\frac 52 \san},
\end{equation*}
for $n$ large enough.
Using \eqref{ini4} and \eqref{bnsize2}, we see that \eqref{redosila}
is justified. Thus, $C_{6,5}=0$. This concludes the proof of Claim
\ref{clazzabi}.
\end{proof}
Now, we are ready to give the final argument of Step 1.

\bigskip


 \noindent - \textit{Final argument of Step 1}.
Focusing on $v_{b,1,0}(s)$, 
we write the following from the table given right before
\eqref{c66}:
\begin{align*}
v_{b,1,0}(\san +\tau)= e^{\frac \tau 2}&\left\{4A^3 B_n^3 C_{4,0}e^{-\san}+3A^4B_n^2e^{-\frac 32 \san}(C_{5,2}+O(B_n))\right.\\
&\left.+2A^5 B_n e^{-2\san}(C_{6,4}+O(B_n))+O\left(\san e^{-\frac 52 \san}\right)\right\}.\nonumber
\end{align*}
Here again, the terms coming with $e^{3\tau/2}$ in the table should
not be taken into account at the formal level, as we have already
explained following that table.
Since $C_{4,0}\neq 0$ from \eqref{c40<0-}, using \eqref{bnsize2}, we see that
\[
v_{b,1,0}(\san +\tau)\sim 4e^{\frac \tau 2}A^3 B_n^3 C_{4,0}e^{-\san}\mbox{ as }n\to \infty,
\]
hence,
a contradiction follows
from the growth factor $e^{\frac\tau 2}$ (the rigorous justification
comes exactly as with the proof of the fact that $C_{6,5}=0$ in Claim \ref{clazzabi}).
This concludes the argument of Step 1, proving that $\aa =
O(\bb^2|\log \bb|)$ as $n\to \infty$.

\bigskip



\textbf{Step 2: Proof of the fact that $C_{6,6}=C_{5,3}=C_{6,5}=0$}.

We will prove these 3 cancelations in the order they appear. Before
doing that, let us recall that the Taylor expansion \eqref{vb1--} still
holds. Even better, we already have some cancelations in that
expansion, namely:  $C_{4,2}=C_{4,1}=0$ and $C_{5,4}=0$
(use Part 1 and Lemma \ref{cl4}).
%
%
One more thing we should keep in mind: using \eqref{noquad} and
Step
1
we know that
\[
\bb^2 \ll \aa \le C \bb^2|\log \bb|\mbox{ as }n\to \infty.
\]
Using the definitions \eqref{defbn-} and \eqref{an2} of $B_n$ and $\san$, we see that
\begin{equation}\label{bnsize3}
Ae^{-\frac \san 2} \ll B_n \le CA\san e^{-\frac \san 2}\mbox{ as }n\to \infty.
\end{equation}
With these facts in mind, we may use the strategy of Proposition
\ref{propexp} and refine the Taylor expansion
\eqref{do}
written for $v_0(y,s)$ and prove that:
\begin{align}
 v_0(y,s)=&-\frac {e^{-2s}}3C_{4,0}^2 \gamma_{4,4,0}h_0h_0
-\frac {e^{-2s}}2C_{4,0}^2 \gamma_{4,4,2}h_2h_0
+\left[C_{4,0}e^{-s}- e^{-2s}C_{4,0}^2 \gamma_{4,4,4}\right]h_4h_0\nonumber\\
&+e^{-\frac 32s}\sum_{j=0}^3C_{5,j}h_{5-j}h_j
+se^{-2s}C_{4,0}^2 \gamma_{4,4,6}h_6h_0
+e^{-2s}\sum_{j=0}^6C_{6,j}h_{6-j}h_j\nonumber\\
&+e^{-2s}C_{4,0}^2h_8h_0+O\left(se^{-\frac 52s}\right),\label{taylor2*}
\end{align}
for some real coefficients $C_{6,j}$ with $j=0,\dots,6$, 
with the same notations following \eqref{taylor2-}. Let us note that
for the formal arguments based on this expansion, the table given
right before
\eqref{c66} 
is valid here.

\bigskip

In the following, we successively prove the 3 cancelations mentioned in the title
of this step.

  \medskip

  \noindent- \textit{
      Proof of the fact that $C_{6,6}=0$}.\\
Using the same strategy as usual, we justify that
\begin{align*}
v_{b,0,0}(\san +\tau) = e^\tau & \left\{A^4 B_n^4 e^{-\san}C_{4,0}+A^5B_n^2e^{-\frac 32 \san}(C_{5,3}+O(B_n))\right.\\
&+\left.A^6 e^{-2\san}(C_{6,6}+O(B_n))+O\left(\san e^{-\frac 52 \san}\right)\right\}.\nonumber
\end{align*}
As we have written following the table given right before \eqref{c66},
we ignore the terms coming with $e^{2\tau}$ in the derivation of this
expansion.\\
Assuming by contradiction that $C_{6,6}\neq 0$ and using \eqref{bnsize3}, we see that
\[
v_{b,0,0}(\san +\tau)\sim e^\tau A^6e^{-2\san}C_{6,6}\mbox{ as }n\to \infty.
\]
Then, the growth factor $e^\tau$ leads to a contradiction (for the
rigorous justification, Proposition \ref{propunif} is enough here,
provided that one takes $A$ large enough; for a similar argument see
the proof of the fact that $C_{6,6}=0$ in Claim \ref{clazzabi} above).
Thus, $C_{6,6}=0$. 

\medskip

\noindent -
  \textit{Proof of the fact that $C_{5,3}=0$}.\\
Using the same strategy as before, and taking into account the size of
$B_n$ given in \eqref{bnsize3} and assuming by contradiction that
$C_{5,3}\neq 0$, we see that
  \[
v_{b,1,0}(\san+\tau) \sim 2 e^{\frac \tau 2}A^4 B_n e^{-\frac 32
  \san}C_{5,3}\mbox{ as } n\to \infty, 
\]
and the growth factor  $e^{\frac \tau 2}$ leads to a contradiction, as
usual. Once again, the terms with $e^{3\tau/2}$ in the table given right
before \eqref{c66} are ignored in this expansion. As for the rigorous
argument, Proposition \ref{propunif} is sufficient for the justification. Thus, $C_{5,3}=0$.

\medskip

\noindent\textit{- Proof of the fact that $C_{6,5}=0$}.\\
By the same strategy, taking into account the size of
$B_n$ given in \eqref{bnsize3} and assuming by contradiction that
$C_{6,5}\neq 0$, we see that
  \[
v_{b,1,0}(\san+\tau) \sim  e^{\frac \tau 2}A^5 e^{-2
  \san}C_{6,5}\mbox{ as } n\to \infty, 
\]
and the growth factor $e^{\frac \tau 2}$ leads to a contradiction, as
before. Note that we ignored the terms involving $e^{3\tau/2}$ in the
table given right before \eqref{c66}. As for the rigorous
justification, it goes as in
the proof of the fact $C_{6,6}=0$ in Claim \ref{clazzabi};
in particular, the power of $A$ is crucial. Thus, $C_{6,5}=0$.
This concludes the proof of the cancelations in the title of Step 2.

\bigskip


\textbf{Step 3: Final argument of Part 4}

Recalling the Taylor expansion \eqref{taylor2*} and the cancelations
of Step 2,
%
% ...
%
we may
use the strategy of Proposition \ref{propexp} and
get to the next order:
\begin{align}
  &v_0(y,s)=
             -\frac {e^{-2s}}3C_{4,0}^2 \gamma_{4,4,0}h_0h_0
-\frac{2e^{-\frac 52s}}3C_{4,0}\sum_{j=0}^1C_{5,j}\gamma_{4,5-j,1-j}h_{1-j}h_j \label{taylor3*} \\
&           -\frac {e^{-2s}}2C_{4,0}^2 \gamma_{4,4,2}h_2h_0
-e^{-\frac 52s}C_{4,0}\sum_{j=0}^2C_{5,j}\gamma_{4,5-j,3-j}h_{3-j}h_j \nonumber\\
&+\left[C_{4,0}e^{-s}- e^{-2s}C_{4,0}^2 \gamma_{4,4,4}\right]h_4h_0
+\sum_{j=0}^2[C_{5,j}e^{-\frac 32s} -2e^{-\frac 52s}C_{4,0}C_{5,j}\gamma_{4,5-j,5-j}] h_{5-j}h_j\nonumber\\
&+se^{-2s}C_{4,0}^2 \gamma_{4,4,6}h_6h_0
                                                            +e^{-2s}\sum_{j=0}^4C_{6,j}h_{6-j}h_j
+\sum_{j=0}^2 2se^{-\frac 52s}C_{4,0}C_{5,j}\gamma_{4,5-j,7-j}h_{7-j}h_j\nonumber\\
&  +\sum_{j=0}^7 C_{7,j}e^{-\frac 52s}h_{7-j}h_j
           +e^{-2s}C_{4,0}^2h_8h_0
  +2C_{4,0}e^{-\frac 52 s}\sum_{j=0}^2 C_{5,j}h_{9-j}h_j             +O\left(s^2e^{-3s}\right), \nonumber
\end{align}
for some real coefficients $C_{7,j}$ with $j=0,\dots,7$, 
with the notations given after \eqref{taylor2-}.

\bigskip

With this expansion, we are ready to find the contradiction which will
show that
\eqref{L00} doesn't hold, concluding thus the
argument of Part 4.\\

\medskip

Starting from the Taylor expansion \eqref{taylor3*} and using the
usual strategy, we may update the tool-table before
\eqref{c66}
and write the following (note that we include only the terms generating 
$e^{\alpha \tau}$ with $\alpha >0$ and that we change the notation by
avoiding to reproduce the coefficients of the first column in the
following; in other words, the second to forth column have to be
multiplied by the first in order to get the desired expression):

\bigskip


\hspace{-3cm}
\begin{tabular}{|c|c|c|c|}
\hline
&$v_{b,0,0}(\san+\tau)$  &$v_{b,1,0}(\san+\tau)$ &$v_{b,0,1}(\san+\tau)$\\
\hline
&&&\\
$C_{4,0}$&$e^\tau A^4 B_n^4e^{-\san}$&$4e^{\frac \tau 2} A^3 B_n^3e^{-\san}$&0\\
&&&\\
$C_{5,i}$&$e^\tau A^5 B_n^{5-i}e^{-\frac 32 \san}$&$(5-i)e^{\frac \tau 2} A^4 B_n^{4-i}e^{-\san}$&$ie^{\frac \tau 2} A^4 B_n^{5-i}e^{-\san}$\\
\hline
&&&\\
$C_{4,0}^2\gamma_{4,4,6}$& $e^\tau A^6 B_n^6 (\san+\tau)e^{-2\san}$ &$6e^{\frac\tau 2}A^5 B_n^5 (\san+\tau)e^{-2\san}$ &0\\
\hline
&&&\\
$C_{6,i}$& $e^\tau A^6 B_n^{6-i} e^{-2\san}$
                         &$(6-i)e^{\frac \tau 2} A^5 B_n^{5-i}
                           e^{-2\san}$ &$ie^{\frac \tau 2} A^5
                                         B_n^{6-i} e^{-2\san}$\\
  \hline
  &&&\\
  $C_{4,0}C_{5,i}\gamma_{4,5-i,7-i}$
  &$e^\tau A^7B_n^{7-i} (\san+\tau)e^{-\frac 52 \san}$
                         &$(7-i)e^{\frac\tau 2}A^6 B_n^{6-i} (\san+\tau)e^{-\frac 52 \san}$
                          &$i e^{\frac\tau 2}A^6B_n^{7-i}  (\san+\tau)e^{-\frac 52 \san}$\\
\hline
  &&&\\
  $C_{7,i}$ &$e^\tau A^7B_n^{7-i}e^{-\frac 52 \san}$
                         &$(7-i)e^{\frac\tau 2}A^6 B_n^{6-i}e^{-\frac 52 \san}$
                          &$i e^{\frac\tau 2}A^6B_n^{7-i} e^{-\frac 52 \san}$\\
\hline
&&&\\
  $C_{4,0}^2$&$e^{2\tau} A^8 B_n^8e^{-2\san}$
                         &$e^{\frac 32\tau} A^7 B_n^7e^{-2\san}$ &0\\
  \hline
  &&&\\
  $C_{4,0}C_{5,i}$
  &$e^{2\tau} A^9 B_n^{9-i} e^{-\frac 52 \san}$
   &$e^{\frac 32 \tau} A^8B_n^{8-i}  e^{-\frac 52 \san}$
  &$e^{\frac 32 \tau} A^8B_n^{9-i}  e^{-\frac 52 \san}$\\                
\hline
&&&\\
$O(s^2e^{-3s})$&$e^\tau O(\san+\tau)^2e^{-3\san}$
&$e^{\frac \tau 2} O(\san+\tau)^2e^{-3\san}$
&$e^{\frac \tau 2} O(\san+\tau)^2e^{-3\san}$\\
\hline
\end{tabular}

 \bigskip

As we have written right after the previous table (given right before
\eqref{c66}), the terms coming with $e^{2\tau}$ and $e^{3\tau /2}$ are
not relevant for our formal argument.

 \medskip
 
From this table, and focusing only on the terms bearing $e^{\frac \tau
  2}$ (remember, this is formal), we may write the following expansion:
\begin{align*}
 & v_{b,1,0}(\san+\tau)
  =e^{\frac \tau 2} \left\{4A^3B_n^3C_{4,0}e^{-\san}
    + 3A^4 B_n^2 e^{-\frac 32 \san}(C_{5,2}+O(B_n))\right.\\
    & \left. +2A^5B_n e^{-2\san}(C_{6,4}+O(B_n))
    +A^6e^{-\frac 52 \san}(C_{7,6}+O(B_n))
  +O(\san^2 e^{-3 \san})\right\}.
\end{align*}
Since $B_n \gg e^{-\frac \san 2}$ by \eqref{noquad} (translated in
particular in \eqref{bnsize3}), we clearly see that
\[
v_{b,1,0}(\san+\tau)
\sim 4e^{\frac \tau 2} A^3B_n^3C_{4,0}e^{-\san}
  \mbox{ as }n\to \infty,
\]
since $C_{4,0}\neq 0$ by \eqref{c40<0-}. Choosing a suitable $\tau$,
we reach a contradiction with \eqref{uij} as usual, at least formally.
Now, on a rigorous level, Proposition \ref{propunif} is not enough, and one needs to
use Proposition \ref{propexp4} and the following argument
to conclude (note that Proposition \ref{propexp4} can indeed be applied here,
since the hypothesis \eqref{rough} is fulfilled thanks to
\eqref{taylor3*}). Therefore, the limit $L_0$ given in \eqref{deflim} is zero.

 \medskip

 As a conclusion to
Subsection \ref{P3},
 whenever we are in the superquadratic case
 \eqref{noquad}, only the $3/2$ power regime is allowed, with a
 constant $L_0$ given by \eqref{L0}. In the next
subsection,
 we will investigate the quadratic regime.
 
\bigskip

\subsection{
The quadratic regime}\label{P4}

This part is devoted to the study of the quadratic regime mentioned on
page \pageref{3cases}. Precisely, we assume that
\begin{equation}\label{quad2}
\aa \sim L\bb^2\mbox{ as }n\to \infty,
\end{equation}
for some $L>0$. We will show that $L$ can enjoy only a finite number
of values, which depend on the Taylor expansion of the solution.

\medskip

Our strategy in the same as
before:
use the Taylor expansion
of $v_0(y,s)$ to derive an expansion for $v_b(y,s)$ with $b=a_n$. The
non-growth condition imposed on the three components $v_{b,0,0}$,
$v_{b,1,0}$ and $v_{b,0,1}$ then implies some conditions on the
coefficients as usual. It happens that those conditions take the form
of polynomial equations on $L$, whose coefficients are derived from
the Taylor expansion.\\
We proceed in 3 steps, where we successively improve the Taylor
expansion of the solution, in order to find the values of $L$.

\medskip

\textbf{Step 1: Expansion of order $e^{-\frac 32 \san}$ in $v_{b,1,0}$}

To begin with, let us recall that
the expansion
\eqref{do}
is valid here.
In particular, the table given right
before \eqref{v1b0t-} will be useful in our argument.
Noting that
\begin{equation}\label{estbn}
B_n \sim LA e^{-\frac \san 2}\mbox{ as }n\to \infty
\end{equation}
from \eqref{quad2} and \eqref{an2}, then,
using our usual strategy,
we prove formally that
\begin{equation*}
v_{b,1,0}(\san +\tau)=e^{\frac \tau 2} A^4 e^{-\frac 32 \san}\left(2C_{4,2}L+C_{5,4}+o(1)\right)
\mbox{ as }n\to \infty.
\end{equation*}
Given the exponential growth factor,
this implies that
\begin{equation}\label{c4254}
2C_{4,2}L+C_{5,4}=0
\end{equation}
(in order to justify rigorously this estimate, simply note that Proposition \ref{propunif} holds here,
since it only uses the fact that
\eqref{do} holds,
which is
the case here; using the ODE argument given after Proposition
\ref{propunif} allows us to conclude).\\
If $C_{4,2}\neq 0$, then $C_{4,2}<0$ from \eqref{discriminant}, and \eqref{c4254} implies that
\[
L=-\frac{C_{5,4}}{2C_{4,2}}.
\]
Since $L>0$, this implies that $C_{5,4}>0$.
Thus, we assume in the following that
\begin{equation}\label{c42*}
C_{4,2}=0.
\end{equation}
Since the multilinear form in
\eqref{defmult}
is nonnegative, this implies that
\begin{equation}\label{c41*}
C_{4,1}=0. 
\end{equation}
Using \eqref{c4254}, we also derive that
\begin{equation}\label{c54*}
C_{5,4}=0.
\end{equation}

\bigskip

\textbf{Step 2: Expansion of order $e^{-2\san}$ in the 3 expanding components}

Starting from the second order expansion
\eqref{do}
and using
\eqref{c42*}, \eqref{c41*} and \eqref{c54*}, we may use the strategy of Proposition \ref{propexp} to derive the following third order Taylor expansion, analogous to \eqref{taylor2-} above (with the same notations):
\begin{align}
 v_0(y,s)=&-\frac {e^{-2s}}3C_{4,0}^2 \gamma_{4,4,0}h_0h_0
-\frac {e^{-2s}}2C_{4,0}^2 \gamma_{4,4,2}h_2h_0
+\left[C_{4,0}e^{-s}- e^{-2s}C_{4,0}^2 \gamma_{4,4,4}\right]h_4h_0\nonumber\\
&+e^{-\frac 32s}\sum_{j=0}^3C_{5,j}h_{5-j}h_j
+se^{-2s}C_{4,0}^2 \gamma_{4,4,6}h_6h_0
+e^{-2s}\sum_{j=0}^6C_{6,j}h_{6-j}h_j\nonumber\\
&+e^{-2s}C_{4,0}^2h_8h_0+O\left(se^{-\frac 52s}\right),\label{taylor3**}
\end{align}
for some real coefficients $C_{6,j}$ with $j=0,\dots,6$.\\
Using our usual strategy, we derive the following expansion:
\begin{equation*}
v_{b,1,0}(\san +\tau) =e^{\frac \tau 2} A^5 e^{-2\san}(2LC_{5,3}+C_{6,5}+o(1)),
\end{equation*}
at least on a formal level (this time, the table right before
\eqref{formal} is useful to derive this,
together with the estimate \eqref{estbn} on $B_n$).
The exponential growth factor
then implies that
\begin{equation}\label{c5365}
2LC_{5,3}+C_{6,5}=0
\end{equation}
(as for the rigorous justification, again, Proposition \ref{propunif}
is sufficient, since the power of $A$ in the linear term in the ODE
solution is $5$, larger that $4$, which is the power of the quadratic contribution).\\
If $C_{5,3}\neq 0$, then we see that
\[
L=-\frac{C_{6,5}}{2 C_{5,3}}.
\] 
Thus, we assume in the following that
\begin{equation}\label{c53*}
C_{5,3}=0.
\end{equation}
Using \eqref{c5365}, we see that
\begin{equation}\label{c65*}
C_{6,5}=0.
\end{equation}
Writing the expansion of $v_{b,0,0}$ this time, we formally see from
\eqref{estbn} and the table given before
\eqref{formal}
that
\[
v_{b,0,0}(\san +\tau) =e^\tau A^6 e^{-2\san}(C_{6,6}+o(1)),
\]
and the growth factor implies that
\begin{equation}\label{c66*}
C_{6,6}=0
\end{equation}
(again, the rigorous justification uses Proposition \ref{propunif} and
the difference of the powers of $A$ between the linear term and the
quadratic term in the solution of the ODE ($6$ against $4$)).

\bigskip

\textbf{Step 3:
Expansion of order $e^{-\frac 52 \san}$ in two expanding components}

At this stage, we need to further improve expansion \eqref{taylor3**}
up to the forth order. Starting from \eqref{taylor3**} and using the
strategy of Proposition \ref{propexp}, together with 
\eqref{c53*}, \eqref{c65*} and \eqref{c66*}, we write a similar
expansion to \eqref{taylor3*}:
\begin{align}
  &v_0(y,s)=
             -\frac {e^{-2s}}3C_{4,0}^2 \gamma_{4,4,0}h_0h_0
 -\frac{2e^{-\frac 52s}}3C_{4,0}\sum_{j=0}^1C_{5,j}\gamma_{4,5-j,1-j}h_{1-j}h_j\nonumber\\
        &   -\frac {e^{-2s}}2C_{4,0}^2 \gamma_{4,4,2}h_2h_0
-e^{-\frac 52s}C_{4,0}\sum_{j=0}^2C_{5,j}\gamma_{4,5-j,3-j}h_{3-j}h_j\nonumber\\
&+\left[C_{4,0}e^{-s}- e^{-2s}C_{4,0}^2 \gamma_{4,4,4}\right]h_4h_0\nonumber\\
&+\sum_{j=0}^2[C_{5,j}e^{-\frac 32s} -2e^{-\frac 52s}C_{4,0}C_{5,j}\gamma_{4,5-j,5-j}] h_{5-j}h_j \label{taylor4*} \\
&+se^{-2s}C_{4,0}^2 \gamma_{4,4,6}h_6h_0
                                                            +e^{-2s}\sum_{j=0}^4C_{6,j}h_{6-j}h_j
  +\sum_{j=0}^2 2se^{-\frac 52s}C_{4,0}C_{5,j}\gamma_{4,5-j,7-j}h_{7-j}h_j\nonumber\\
  &+\sum_{j=0}^7 C_{7,j}e^{-\frac 52s}h_{7-j}h_j
           +e^{-2s}C_{4,0}^2h_8h_0
 +2C_{4,0}e^{-\frac 52 s}\sum_{j=0}^2 C_{5,j}h_{9-j}h_j             +O\left(s^2e^{-3s}\right),\nonumber
\end{align}
for some real coefficients $C_{7,j}$ with $j=0,\dots,7$, 
with the notations given after \eqref{taylor2-}.

\medskip


Starting from the table in
Step 3 of Part 4
above, together with estimate
\eqref{estbn} on $B_n$, we formally obtain the following, thanks to our usual strategy:
\begin{align}
v_{b,0,0}(\san+\tau)&=e^\tau A^7 e^{-\frac 52 \san}(C_{7,7}+o(1)),\label{tony}\\
v_{b,1,0}(\san+\tau)&=e^{\frac \tau 2}A^6 e^{-\frac 52 \san}(4L^3 C_{4,0}+3L^2 C_{5,2}+2LC_{6,4}+C_{7,6}+o(1)).\nonumber
\end{align}
The growth factors both imply that (at least formally)
\begin{equation}\label{c77*}
C_{7,7}=0\mbox{ and }4L^3 C_{4,0}+3L^2 C_{5,2}+2LC_{6,4}+C_{7,6}=0.
\end{equation}
In particular, we see that $L$ solves an equation of degree 3
(remember that $C_{4,0}\neq 0$ from \eqref{discriminant}). Thus, it can at most bear 3 values. We
need of course to justify \eqref{c77*} rigorously.
For that, we need to refine Proposition \ref{propexp4} up to the
order $e^{-\frac 32  \san}$. More precisely, this is our statement:
\begin{prop}[Uniform estimate of $v_{a_n}(y,\san)$ up to the order
  $e^{-\frac 32  \san}$]\label{propexp5}
  With all the information we have in this Step 3 of Subsection
  \ref{P4}, 
  we claim that
 that for all $A\ge 1$ and $\dd \ge 1$, for $n$ large enough, we have
 for all
$s\in [\san, \dd \san]$,
\[
  \|v_{a_n}(y,s)-C_{4,0}e^{-s}h_4(y_1)\|_{L^2_\rho}\le
CA^2e^{-\frac 32 \san}
\]
\end{prop}
\begin{proof}
See Subsection \ref{secunif}.
\end{proof}
As one may see by comparing the 2 propositions, the new version
replaces the $C(A)$ constant appearing in Proposition
\ref{propexp4} by $CA^2$. With this improvement, the argument given right after
that proposition works (thanks to the difference between the powers of
$A$), confirming the validity of \eqref{c77*}.
Note however that expansion \eqref{tony} is valid for $\tau$ large,
but not for $\tau=0$; for that reason, one has to carefully use
\eqref{taylor4*} to derive that
\[
  v_{b,0,0}(\san) =- \frac{e^{-2\san}}3C_{4,0}^2\gamma_{4,4,0}
  + e^{-\frac 52 \san}(-\frac 23 AC_{4,0}C_{5,1}\gamma_{4,4,0}
  +A^7C_{7,7} +o(1)).
\]
Apart from that, the adaptation of the argument given right after
Proposition \ref{propexp4} is straightforward.

\medskip

This concludes the proof of Theorem \ref{cor0m} when $m=4$, if one assumes the
result of  Propositions \ref{propunif},
\ref{propexp4} and \ref{propexp5}. In order to finish the argument,
those propositions will be proved in the next subsection.

\subsection{Uniform estimates for $w_{a_n}$}\label{secunif}

This subsection is dedicated to the proofs of Propositions \ref{propunif},
\ref{propexp4} and \ref{propexp5}.
This
is the only missing element
to terminate the proof of Theorem \ref{cor0m} in the case $m=4$, given in the previous
subsections.

\begin{proof}[Proof of Proposition \ref{propunif}]
We will in fact sketch the idea, and avoid giving details.\\
Let us first start by recalling from Khenissy, Rebai and Zaag \cite{KRZihp11} that 
\begin{equation}\label{sim24}
v_{a_n}(y,s) \sim P_{\{2\le i \le 4\}}(v_{a_n}) \mbox{ as } s\to \infty,
\end{equation}
uniformly for $n$ large enough (see Theorem 1 page 4 in that paper).
We also know from \eqref{uniliou} that
\[
\|v_{a_n}(s)\|_{L^2_\rho}\to 0\mbox{ as } s\to \infty,
\]
uniformly in $n$.
From \eqref{sim24}, our desired conclusion will follow if we 
fix $\cc_0>0$ such that for any $A\ge 1$ and $\dd \ge 1$, 
\begin{equation}\label{naila}
\forall i=2,3,4,\;\;\forall j=0,\dots,i,\;\;
\forall s\in [\san, D\san],
\;\;|v_{a_n,i,j}(s)|\le \cc_0 A^2 e^{-\san},
\end{equation}
whenever $n$ is large enough.\\
We proceed in 2 steps to fix such a $\cc_0$:
in Step 1, we initialize \eqref{naila} at $s=\san$, then, in Step 2, we proceed by contradiction to prove it for $s\ge \san$ (the integration of the projections of equation \eqref{eqv} on the various coordinates will be crucial for our argument).\\
\textit{Step 1: Initialization of \eqref{naila}}. Since there is a universal constant $C_2\ge 1$ such that for any $v\in L^2_\rho$ and $i=2,3,4$, 
\begin{equation}\label{equiv}
\frac 1{C_2} \max_{j=0,\dots,i}|v_{i,j}|\le \|P_{i}(v)\|_{L^2_\rho}\le C_2 \max_{j=0,\dots,i}|v_{i,j}|,
\end{equation}
using \eqref{fai} together with \eqref{sim24}, we see that 
\begin{equation}\label{fai2}
\forall i=2,3,4,\;\;\forall j=0,\dots,i,\;\;
|v_{a_n,i,j}(\san)|\le \bar C A^2 e^{-\san},
\end{equation}
whenever $n$ is large enough, for some universal constant $\bar C>0$.\\
Fixing
\begin{equation}\label{defc0}
\cc_0 =2 \bar C,
\end{equation}
we guarantee from \eqref{fai2} that
\eqref{naila} holds at $s=\san$.\\
In particular, \eqref{naila} holds at $s=\san$.\\
\textit{Step 2: The contradiction argument}. Let us assume by
contradiction that \eqref{naila} is true for all $s\in [\san, s^*_n]$,
for some minimal $s^*_n< \dd \san$ and stops from being true at $s=s^*_n$.
From Step 1 and continuity, it follows that $s_n^*>\san$. This also implies
that we have an equality case at $s=s^*_n$, in the sense that
\begin{equation}\label{hakan}
|v_{a_n,i,j}(s_n^*)|= \cc_0A^2 e^{-\san},\mbox{ for some }i=2,3,4\mbox{ and } j=0,\dots,i.
\end{equation}
Starting from
\eqref{defc0},
we integrate the ODE
\eqref{eqvij}
satisfied by $v_{a_n,i,j}$, with $2 \le i \le 4$:
\[
\forall s\ge s_n,\;\;v_{a_n,i,j}'(s)=(1-\frac i2) v_{a_n,i,j}(s)+\int k_{i-j}(y_1)k_j(y_2) v_{a_n}(y,s)^2\rho(y) dy. 
\]
Since \eqref{naila} is valid for all $s\in [\san, s^*_n]$, using \eqref{sim24}, we have
\[
|\int k_{i-j}(y_1)k_j(y_2) v_{a_n}(y,s)^2\rho(y) dy|\le C\cc_0^2A^4e^{-2\san}
\]
(in fact, we need here an estimate in the $L^4_\rho$ norm and not just
in the $L^2_\rho$, and this is possible thanks to parabolic
regularity, as we have already explained right after the statement of
Proposition \ref{propunif}).
Therefore, for all $s\in [\san, s^*_n]$, we have
\[
|v_{a_n,i,j}'(s)-(1-\frac i2) v_{a_n,i,j}(s)|\le C\cc_0^2A^4e^{-2\san}.
\]
Recalling that $s_n^*\le \dd\san$, we integrate this differential
inequality, ending with
\begin{align}
|v_{a_n,i,j}(s)|&\le e^{(1-\frac i2)(s-\san)}|v_{a_n,i,j}(\san)|+C\cc_0^2A^4e^{-2\san}(s-\san),\nonumber\\
 &\le \frac {\cc_0}2 A^2e^{-\san}+C\cc_0^2A^4e^{-2\san}(\dd-1)\san\le \frac {3\cc_0}4A^2 e^{-\san},\label{est34}
\end{align}
if $n$ is large enough. This means in particular that \eqref{hakan} is
not true and concludes the proof of Proposition \ref{propunif}.
\end{proof}

\begin{proof}[Proof of Proposition \ref{propexp4}]
Following Proposition \ref{propunif}, we assume that \eqref{rough} holds.
We also recall from \eqref{defbn-} that $B_n \to 0$ as $n\to \infty$.\\
We proceed in 2 steps: we first initialize the estimate at $s=\san$,
then, we use the bound in Proposition \ref{propunif} to derive then integrate differential inequalities satisfied by the various components.\\ 
\textit{Step 1: Initialization at $s=\san$}. Using the expansion
\eqref{rough} and using our geometric transformation as usual, 
we may write the following estimates for the various components of the
solution at $s=\san$ (for the last estimate, remember that
\eqref{rough} holds also  in
$L^4_\rho$ and use the Cauchy-Schwarz identity as we did for
\eqref{vbar2}):
\begin{align}
v_{a_n,4,0}(\san)&=C_{4,0}e^{-\san}+O(e^{-\frac 32 \san}),\label{ini1}\\
|v_{a_n,i,j}(\san)|&\le C_ie^{-\frac 32 \san}\mbox{ if }(i,j) \neq (4,0),\label{ini2}\\
\|P_{\{\lambda \le -\frac 32\}}(v_{a_n}(\san))\|_{L^2_\rho}&\le C(A)e^{-\frac 32 \san}.\label{ini3}
\end{align}
\textit{Step 2: Differential inequalities satisfied by the various
  components}. Using Proposition \ref{propunif} whose conclusion also
holds in $L^4_\rho$, as explained in the remark following its
statement, 
we may write
the following differential inequalities, for $n$ large enough and for
all
$s\in[\san, \dd \san]$,
\begin{align*}
\forall i\in \m N,\;\;\forall j=0,\dots,i,\;\;
|v_{a_n,i,j}'(s)-(1-\frac i2)v_{a_n,i,j}(s)|&\le C_i A^4 e^{-2\san},\\
\frac d{ds}\|P_{\{\lambda \le 1-\frac i2\}}(v_{a_n}(s))\|_{L^2_\rho}^2\le 
-2(1-\frac i2)\|P_{\{\lambda \le 1-\frac i2\}}(v_{a_n}(s))\|_{L^2_\rho}^2
+&C_i^2 A^8 e^{-4\san}.
\end{align*}
Integrating these various differential inequalities, we get the
following estimates, for $n$ large enough and for all
$s\in[\san, \dd \san]$,
\begin{align}
\forall i=0,1,\;\;\forall j=0,\dots,i,\;\;
|v_{a_n,i,j}(s)|&\le CA^4 e^{-2\san},\label{a1}\\
\forall j=0,1,2,\;\;|v_{a_n,2,j}(s)-v_{a_n,2,j}(\san)|&\le C(s-\san)A^4 e^{-2\san}, \label{a2}\\
\forall i=3,4,\;\;\forall j=0,\dots,i,\;\;
|v_{a_n,i,j}(s)-e^{(1-\frac i2)(s-\san)}v_{a_n,i,j}(\san)|&\le CA^4 e^{-2\san}, \label{a3}\\
\|P_{\{\lambda \le -\frac 32\}}(v_{a_n}(s))\|_{L^2_\rho}\le 
e^{-\frac32(s-\san)}\|P_{\{\lambda \le -\frac 32\}}(v_{a_n}(\san))\|_{L^2_\rho}&+CA^4 e^{-2\san}. \label{a4}
\end{align}
Using \eqref{ini1}-\eqref{ini3}, and recalling
that $s\le \dd \san$, 
we conclude the proof of Proposition \ref{propexp4}.
\end{proof}

\begin{proof}[Proof of Proposition \ref{propexp5}]
  Note first that we can use all the information available in Step 3 of
  Subsection
\ref{P4}
  For the proof, we follow the proof of Proposition \ref{propexp4} with only one
modification: we need to refine identities \eqref{ini1}-\eqref{ini3} so that we
reach the orders $e^{-\frac 32 \san}$. For that reason, we will keep only the
$e^{-s}$ and $e^{-\frac 32 s}$ terms in the expansion \eqref{taylor4*}
and write
\[
  v_0(y,s)=C_{4,0}e^{-s} h_4h_0
  +e^{-\frac  32s}\sum_{j=0}^2C_{5,j}h_{5-j}h_j+O(se^{-2s})
\]
as $s\to \infty$, in $L^q_\rho(\m R^N)$, for any $q\ge 2$.
  Then, using the estimate \eqref{estbn} on $B_n$ together with the
change of variables \eqref{wbw-} (with $\tau=0$), we may use the
binomial relation \eqref{binomial} to write
\begin{align*}
  v_b(y_b, \san) =
  &C_{4,0}e^{-\san}\left[h_4(y_{b,1})+4h_3(y_{b,1})AB_n+O(B_n^2)\right]\\
  &+e^{-\frac
    32\san}\sum_{j=0}^2C_{5,j}h_{5-j}(y_{b,1})h_j(y_b,2)+O(C(A)\san e^{-2\san})
\end{align*}
in $L^2_\rho(\m R^N)$ as $n\to \infty$, thanks to the same
justification as for \eqref{vbar2}. This can be translated as follows:
\begin{align*}
  v_{b,3,0}(\san) &= 4C_{4,0}LA^2e^{-\frac 32 \san}+o(e^{-\frac 32 \san}),\\
  v_{b,4,0}(\san)&= C_{4,0}e^{-\san} +O\left(C(A)\san e^{-2\san}\right),\\
  \mbox{for }i=0,1,2\;\; v_{b,5,i}(\san)
   &= C_{5,i}e^{-\frac 32 \san}+O\left(C(A)\san e^{-2\san}\right),\\
  \sup_{(i,j) \not\in\{(3,0),(4,0), (5,0), (5,1), (5,2)\}}|v_{b,i,j}(\san)|&+\|P_{\{\lambda \le -\frac 32\}}(v_{a_n}(\san))\|_{L^2_\rho}
 \le C(A)\san e^{-2 \san}
\end{align*}
Using these estimates as initial data, together with the estimates
\eqref{a1}-\eqref{a4}, which hold here, since the hypotheses of
Proposition \ref{propexp4} hold also (because \eqref{rough} follows
from \eqref{taylor4*}, which holds like all the information of Step 3
in Subsection
\ref{P4},
we conclude the proof of
Proposition \ref{propexp5} (we also need Proposition \ref{propunif}).
\end{proof}

\medskip

We have just proved the  Propositions \ref{propunif},
\ref{propexp4} and \ref{propexp5} in this subsection, finishing this
way the proof of Theorem \ref{cor0m}.
\section{Rigidity
in the geometry of the blow-up set when
   $m\ge 6$}\label{sec0m}
 This section is devoted to the proof of Theorem \ref{cor0m} when $m\ge
 6$. As for $m=4$, the
 proof uses the same strategy based on the geometric transformation \eqref{wbw000}
 given in Step 2 of the proof of Theorem \ref{cor0m}.
 However, the number of steps depends linearly on
$m$. For that reason, we have to proceed differently, at some point in
the proof. In addition, we have more complicated formulas, and of
course in the outcome, the result for $m\ge 6$ is less explicit,
hence, less spectacular
than for $m=4$.
Accordingly, we only give the main steps of the proof and don't insist on details. 

\medskip

\begin{proof}[Proof of Theorem \ref{cor0m} when $m\ge 6$]
 Consider $u(x,t)$ a
  solution of equation \eqref{equ} blowing up at
  time $T>0$.
 Assume that the origin is a
  non isolated blow-up point where $m(0)=m\ge 6$ is even and consider an arbitrary
  sequence of non-zero blow-up points $a_n=(a_{n,1},a_{n,2})$ converging to
  the origin as $n\to \infty$.\\
  (i) This item follows exactly as in the case $m=4$ (see Subsection
  \ref{P2} above). In particular, after extracting a subsequence and making a suitable change of variables, 
we may assume that 
\begin{equation}\label{posm}
\aa\ge 0\mbox{ and }a_{n,2}\ge 0\mbox{ for all }n\in \m N
\end{equation}
and
\begin{equation}\label{onem}
\aa=o(\bb).
\end{equation}
Note that Theorem \ref{th0m} holds here, and so does the expansion
\eqref{do}, and that the multilinear form
in \eqref{defmult}
is non
zero and nonnegative.\\
(ii) Following \eqref{onem}, after the extraction of a subsequence
still denoted the same, we identify 3 possible regimes for $a_n$: superquadratic, quadratic and subquadratic,
as with $m=4$.
We then proceed in 3 parts:\\
- Part 1 is dedicated to the superquadratic case, where $\aa \gg
\bb^2$ as $n\to \infty$.\\
- In Part 2, we deal with the quadratic case, where $\aa \sim L \bb^2$,
for some $L>0$.\\
- Finally,
in Part 3,
we gather all the information and conclude the
proof of Theorem \ref{cor0m}.

\bigskip

\textbf{Part 1: Quantified superquadratic regimes for $a_n$} 

In this part, we assume that
\begin{equation}\label{noquadm}
a_{n,1}\gg a_{n,2}^2\mbox{ as }n\to \infty.
\end{equation}
We then proceed in
4
steps:\\
- In Step 1, we show that  $C_{m,m-2}=C_{m,m-3}=0$, where those
coefficients appear in the expansion \eqref{do}.\\
- In Step 2,
we introduce a new parameter $\theta$ to
measure the convergence of $B_n=a_{n,1}/a_{n,2}$ to $0$ in exponential
scales of $\san$ defined in \eqref{an2}, and show that $\theta$ enjoys
only a finite number of values in $\m Q\backslash\{0\}$.\\
- In Step 3, we make a refinement of Step 2, by introducing one
further parameter $\alpha$ to measure the convergence of $B_n$ to $0$ in
polynomial corrections of exponential scales, and show that $\alpha$
enjoys only a finite number of values in $\m Q$.\\
- In Step 4, following Steps 2 and 3, we make one further refinement,
by introducing a new variable $\psi_n$ to quantify the convergence of
$B_n$ to $0$, and show that it converges (up to a subsequence) to
some value related to some root of a polynomial involving
the coefficients of the Taylor expansion of $v_0(y,s)$.

\bigskip

\textbf{Step 1: Proof of the fact  that $C_{m,m-2}=C_{m,m-3}=0$}

Considering the expansion \eqref{do}, we will now prove that
$C_{m,m-2}=C_{m,m-3}=0$.
As in the case $m=4$, the nonnegativity of the multilinear form
\eqref{defmult}
implies that it is enough to show that $C_{m,m-2}=0$.
Proceeding by contradiction, we assume that $C_{m,m-2}\neq 0$. Let us
reach a contradiction from the behavior of $w_b$ defined by
\eqref{wbw-}, where $b=a_n$. Using the notations and assumptions in
\eqref{defbn-}-\eqref{an2}, we see from \eqref{noquadm}, \eqref{defrn-} and \eqref{onem} that
\begin{equation}\label{sizebn0}
1\gg B_n \gg Ae^{-\frac \san 2}\mbox{ as }n\to \infty.
\end{equation}
Using the expansion
\eqref{do} and proceeding as in Part 1 of Subsection \ref{P3} when
$m=4$, both formally and rigorously,
we readily see that
\[
v_{b,1,0}(\san+\tau) \sim 2e^{\frac \tau 2}A^{m-1}B_nC_{m,m-2}e^{(1-\frac m2)\san}\mbox{ as }n\to \infty.
\]
The growth factor leads to contradiction. Thus, $C_{m,m-2}=0$ and
$C_{m,m-3}=0$ too.

\bigskip

\textbf{Step 2:  Measuring $B_n$ in exponential scales of $\san$}.

From
\eqref{sizebn0}, we may assume that 
\begin{equation}\label{deftheta}
  -\frac{\log B_n}{\san} \to \theta \in [0, \frac 12]
\mbox{ or equivalently that }B_n = e^{-(\theta+o(1))\san},
  \mbox{ as }n\to \infty,
\end{equation}
up to a extracting a subsequence, still denoted the same.
Now,
we are going to prove the following three-fold statement:
\begin{prop}\label{proptheta} $ $\\
  (i) The limit $\theta \in [0,\frac 12]$ defined in \eqref{deftheta} is rational.\\
(ii)
  It holds that $\theta\neq 0$.\\
(iii)
  In fact, $\theta$ enjoys only a finite number of rational
  values in $E_1\cup E_2$ where the two sets $E_1$ and $E_2$ are
  defined below respectively in \eqref{defE1} and \eqref{defE2}.
 \end{prop}
\begin{nb}
  In fact, as we will see from the proof, items
(ii) and (iii)
  are by-products of
 the proof of item (i). For clarity, we dedicate the next step to the
 proof of item (i), then, we explain how to derive items
(ii) and (iii)
 in the following steps. 
\end{nb}
\begin{nb}
  Note that
  \[
E_1 \cup E_2 \subset \{\frac L{2G}\;\; 1\le G\le 2m-3\mbox{ and }1\le L\le \min(G,m-2) \}.
  \]
\end{nb}
\begin{proof}
  We will give the proof of the three items successively.

  \medskip

  \noindent - \textit{Proof of item (i) of Proposition \ref{proptheta}}. 

 Using the strategy of Proposition \ref{propexp}, we may refine the
Taylor expansion of $v_0(y,s)$ (equal to $w_0(y,s)-1$ by definition
\eqref{defv}) given in
\eqref{do} (recalling that  $C_{m,m-2}=C_{m,m-3}=0$ from Step 1),
up to the
order $O(se^{(2-m)s})$:
\begin{align}
 v _0(y,s)=& e^{(1-\frac m2)s}\sum_{j=0}^{m-4}C_{m,j}h_{m-j}(y_1)h_{j}(y_2) 
           +e^{\frac
             {1-m}2s}\sum_{j=0}^mC_{m+1,j}h_{m+1-j}(y_1)h_{j}(y_2)
            \nonumber\\
           &+e^{(1-\frac k2)s}\sum_{k=m+2}^{2m-3}\sum_{j=0}^k C_{k,j}h_{k-j}(y_1)h_j(y_2)
              +\bar v_0(y,s) \label{expv0max}
\end{align}
with $\bar v_0(y,s)=O(se^{(2-m)s})$ as $s\to \infty$.
Note
that we stopped in this expansion when the quadratic term in the
equation \eqref{eqv} satisfied by $v_0(y,s)$ becomes relevant; this
way, we only have ``linear terms'' in our expansion.\\
Using the transformation \eqref{wbw000}, we may use the
expansion in
\eqref{expv0max}
to derive the
following expansion for $v_b(y_b,\san+\tau)$, which is analogous to
\eqref{vb1--}, with $\tau\ge 0$:
\begin{align}
v_b(y_b,\san+\tau)&=e^{(1- \frac
                    m2)(\san+\tau)}\sum_{j=0}^{m-2}C_{m,j}h_{m-j}(y_{b,1}+AB_ne^{\frac \tau 2})h_{j}(y_{b,2}+Ae^{\frac\tau 2})\nonumber\\
&+e^{\frac {1-m}2(\san+\tau)}\sum_{j=0}^mC_{m+1,j}h_{m+1-j}(y_{b,1}+AB_ne^{\frac\tau 2})h_{j}(y_{b,2}+Ae^{\frac\tau 2})   \nonumber\\
           &+e^{(1-\frac
             k2)(\san+\tau)}\sum_{k=m+2}^{2m-3}\sum_{j=0}^k
             C_{k,j}h_{k-j}(y_{b,1}+AB_ne^{\frac\tau
             2})h_j(y_{b,2}+Ae^{\frac\tau 2})\nonumber\\
  &+\bar v_b(y_b,\san+\tau).\label{vb20}
\end{align}
As before, for more visibility, we may write the following table
giving the expansion for the 3 expanding components $v_{b,0,0}(\san+\tau)$,
 $v_{b,0,1}(\san+\tau)$ and $v_{b,1,0}(\san+\tau)$, for $\tau\ge 0$:

\bigskip

\hspace{-38mm}
\begin{tabular}{|c|c|c|c|}
\hline
&$v_{b,0,0}(\san+\tau)$  &$v_{b,1,0}(\san+\tau)$ &$v_{b,0,1}(\san+\tau)$\\
\hline
&&&\\
  $C_{m,j}$&$C_{m,j}e^\tau A^m B_n^{m-j}e^{(1-\frac  m2)\san}$
   &$(m-j)C_{m,j}e^{\frac \tau 2}A^{m-1} B_n^{m-1-j}e^{(1-\frac  m2)\san}$
   &$jC_{m,j}e^{\frac \tau 2} A^{m-1} B_n^{m-j}e^{(1-\frac   m2)\san}$\\
\hline
&&&\\
$C_{m+1,j}$&$C_{m+1,j}e^\tau A^{m+1} B_n^{m+1-j}e^{\frac {1-m}2 \san}$&$(m+1-j)C_{m+1,j} e^{\frac \tau 2}A^m B_n^{m-j}e^{\frac {1-m}2 \san}$&$jC_{m+1,j}e^{\frac \tau 2} A^m B_n^{m+1-j}e^{\frac {1-m}2 \san}$\\
\hline
&&&\\
$C_{k,j}$&$C_{k,j}e^\tau A^k B_n^{k-j}e^{(1-\frac
           k2)\san}$&$(k-j)C_{k,j}e^{\frac \tau 2}A^{k-1} B_n^{k-1-j}e^{(1-\frac
                      k2)\san}$&$jC_{k,j}e^{\frac \tau 2} A^{k-1} B_n^{k-j}e^{(1-\frac
                                 k2)\san}$\\
\hline
&&&\\
$\bar v_b(y_b,\san+\tau)$ &$\bar v_{b,0,0}(\san+\tau)$&$\bar v_{b,0,1}(\san+\tau)$  &$\bar v_{b,1,0}(\san+\tau)$\\ 
\hline
\end{tabular}

\bigskip

\noindent with the rest terms satisfying
\begin{equation}\label{sizerest}
|\bar v_{b,i,j}(y_b,\san)|\le C(A,i,j) \san e^{(2-m) \san},
\end{equation}
with the same argument as for \eqref{vb00}.

\medskip

Now, proceeding by contradiction, we assume that $\theta$ is not
rational. Let us first explain our argument. As in the case $m=4$, the
contradiction will follow from the behavior of one of the 3 components
shown in the table above.
In fact, since the multilinear form
in \eqref{defmult}
is non zero,
there is $l=0,\dots,m$ such that $C_{m,l}\neq 0$. Therefore, it is
convenient to choose a component involving this $C_{m,l}$, hoping to
reach a contradiction. Since $l\in[0,m-4]$
from \eqref{do} and Step 1,
we focus on the first two components, since the third
misses $C_{m,0}$. More precisely, we will choose the second component,
namely $v_{b,1,0}(\san+\tau)$, since it involves lower powers of the small parameter $B_n$. Furthermore, we
need the coefficient of the term involving $C_{m,l}$, namely $A^{m-1}e^{\frac \tau 2} B_n^{m-1-l}e^{(1-\frac m2)\san}$, to be dominant with respect to the
error term whose size is shown in \eqref{sizerest}.
Since we have from \eqref{deftheta} that
\begin{equation}\label{nadin}
  A^{m-1}e^{\frac \tau 2} B_n^{m-1-l}e^{(1-\frac m2)\san}
\ge  A^{m-1}e^{\frac \tau 2} B_n^{m-1}e^{(1-\frac m2)\san}
  =
  A^{m-1}e^{\frac \tau 2}e^{(1-\frac
    m2-(m-1)\theta+o(1))\san}
  \end{equation}
 as $n\to \infty$, a sufficient condition for this is to have 
\begin{equation}\label{gregoire}
1-\frac m2 - (m-1)\theta >2-m, 
\mbox{ i.e. } \theta <\frac{m-2}{2(m-1)}.
\end{equation}
Accordingly, since $\theta\in[0,\frac 12]$ by \eqref{deftheta}, we
consider two cases in the following:

\medskip

\textbf{Case 1}: $\theta \in[0,\frac{m-2}{2(m-1)})$.
From the table
above together with \eqref{sizerest}, we consider the following
expansion:
\begin{equation}\label{sol}
  v_{b,1,0}(\san+\tau) = e^{\frac \tau 2}\left(
    \sum_{(k,j)\in H_1}e^{(1-\frac  k2)\san}A^{k-1}
    (k-j)C_{k,j}B_n^{k-1-j}
    +O(\san e^{(2-m)\san})\right),
\end{equation}
where
\begin{align}
  H_1=\{(k,j)\;\;|\;\;&m\le k \le 2m-3, \;\;0\le j\le k-1,\;\;\label{defh1}\\
  &C_{k,j}\neq
  0\mbox{ and }e^{(1-\frac  k2)\san}B_n^{k-1-j} \gg \san e^{(2-m) \san}\}.\nonumber
\end{align}
We first note that $H_1\neq \emptyset$, since it contains $(m,l)$, by
\eqref{nadin} and \eqref{gregoire}.
Second, the following non-codominance property between all
terms in the expansion \eqref{sol} allows us to conclude:
\begin{lem}[No codominance of terms in the expansion of
  $v_{b,1,0}(\san + \tau)$] \label{lemnocodom} Consider two terms  in
  the expansion
  \eqref{sol},
  say $(k-j)C_{k,j}A^{k-1} B_n^{k-1-j}e^{(1-\frac
    k2)\san}$ and
  $(k'-j')C_{k',j'}A^{k'-1} B_n^{k'-1-j'}e^{(1-\frac
    {k'}2)\san}$ with $C_{k,j}C_{k',j'}\neq 0$, $m\le k,k'\le 2m-3$, $0\le j\le k-1$, $0\le
  j'\le k'-1$, and  $(k,j) \neq(k',j')$. Then, one of the terms
  dominates the other as $n\to \infty$.
\end{lem}
\noindent Indeed, if this lemma holds, recalling that $H_1$ is a non empty finite set, we may consider $(\bar k, \bar j)\in H_1$ such that
\begin{equation}\label{cesaria}
v_{b,1,0}(\san+\tau) \sim e^{\frac \tau 2}(\bar k-\bar j)C_{\bar
  k,\bar j}A^{\bar k-1} B_n^{\bar k-1-j}e^{(1-\frac
    {\bar k}2)\san}
\end{equation}
as $n\to \infty$, with $C_{\bar k, \bar j}\neq 0$.
From the growth factor $e^{\frac \tau 2}$, the coordinate
$v_{b,1,0}(\san+\tau)$ will grow, and a contradiction follows as
usual, both for the formal and the rigorous argument. It remains then to prove
 Lemma \ref{lemnocodom}.
\begin{proof}[Proof of Lemma \ref{lemnocodom}]
  Using the expansion of $B_n$ in \eqref{deftheta}, we see that
  \begin{equation}\label{thomas}
(k-j)C_{k,j}A^{k-1} B_n^{k-1-j}e^{(1-\frac  k2)\san}
= (k-j)C_{k,j}A^{k-1}  e^{(1-\frac
  k2-(k-1-j)\theta+o(1))\san}
\end{equation}
as $ n\to \infty$,
  with a similar expansion with $(k',j')$. Since $(k,j) \neq(k',j')$,
  it follows that $(k,k-j) \neq(k',k'-j')$. Therefore, recalling that $\theta$ is not
  rational (this is in fact the contradiction hypothesis), it follows that
  \[
1-\frac k2-(k-1-j)\theta \neq 1-\frac {k'}2-(k'-1-j')\theta,
\]
  and the conclusion follows.
\end{proof}
\begin{nb} From this proof, we see that the hypothesis that $\theta$
  is not rational is too strong. In fact, the argument works whenever
  $\theta$ avoids the rationals which are of the form
  $\frac{k'-k}{2[(k-j)-(k'-j')]}$, which make a finite collection of
  numbers, due to the boundedness of the ranges where the parameters
  lay. This remark will show to be crucial
below,
  while adapting the present
  step in order to derive the proof of item
(iii)
  of Proposition \ref{proptheta}.
\end{nb}

\medskip

\textbf{Case 2}: $\theta \in[\frac{m-2}{2(m-1)}, \frac 12]$. It
happens that the argument of Case 1 works here, with small natural
adaptations.\\
The first problem is that the terms in the expansion \eqref{sol} may
all be dominated by the error term, therefore, we need to go
further in the Taylor expansion \eqref{expv0max}, and this is possible
thanks to Proposition \ref{propexp}.
The question then is to know  the order 
up to which we carry
  on the Taylor expansion.

  \medskip

  To find  that, note that focusing on terms with  $C_{k,j}\neq 0$ is a convenient way to have a relevant term in the expansion of $v_{b,1,0}(\san+\tau)$.
  Such a term exists with $(k,j)=(m,l)$, since
  we know that the multilinear form in
\eqref{defmult}
is non-zero. It remains then to guarantee that the order in front of
$C_{m,l}$, namely $e^{(1-\frac m2)\san}B_n^{m-1-l}$, is dominant with
respect to the error term. Since $l\ge 0$, hence $e^{(1-\frac
  m2)\san}B_n^{m-1-l}\ge  e^{(1-\frac m2)\san}B_n^{m-1}=
e^{(1-\frac m2 - (m-1) \theta+o(1))\san}$ as $n\to \infty$ by \eqref{deftheta},
 we simply need to refine the 
the Taylor expansion \eqref{expv0max} of $v_0(y,s)$ up to the order
$O(s^\gamma e^{(1-\frac{M+1}2)s})$ for some $\gamma>0$, with $M\in \m
N$ given by
\begin{equation}\label{defM}
M= \lceil (2\theta+1)(m-1)\rceil,
\end{equation}
where the notation $\lceil \rceil$ stands for the ceiling (or upper
integer part) of a given number. This is possible,
thanks to Proposition \ref{propexp}.
However, in comparison with the Taylor
expansion given in \eqref{expv0max}, now we will see ``resonant''
terms, of order $s^i e^{(1-\frac k2)s}$, the first among them occurs
  at $k=2m-2$, and corresponds to the effect of the quadratic term in
  \eqref{eqv}. More precisely, that term is of order $se^{2-m}$. 
  \medskip

  Now, using the geometric transformation given in Step 2
of Section \ref{sectaylor},
  we may write an expansion of $v_b(y_b,\san+\tau)$ analogous to
  \eqref{vb20}. We may also write a table similar to the one right before
  \eqref{sizerest}, giving an expansion for $v_{b,1,0}(\san+\tau)$
  analogous to \eqref{sol}, and which shows resonant orders, as
  follows:
 \begin{align}
  v_{b,1,0}(\san+\tau) = e^{\frac \tau 2}\left(
    \sum_{(k,j,i)\in H_2}\right.&(\san+\tau)^i e^{(1-\frac  k2)\san}A^{k-1}
    (k-j)\tilde C_{k,j,i}B_n^{k-1-j}\nonumber\\
                                &\left.+ O(\san^\gamma e^{(\frac{1-M}2)\san})\right), \label{sol2}
\end{align}
where $H_2\subset \m N^3$ is a natural adaptation of the set $H_1$ 
\eqref{defh1}
as follows:
\begin{align}
  H_2=\{(k,j,i)\;\;|\;\;&m\le k \le M, \;\;0\le j\le k-1,\;\;
  0\le i \le i_k,\label{defh2}\\
  &\tilde C_{k,j,i}\neq
  0\mbox{ and }\san^i e^{(1-\frac  k2)\san}B_n^{k-1-j} \gg \san^\gamma e^{(\frac{1-M}2) \san}\},\nonumber
\end{align}
where $i_k \in \m N$ is positive only at resonant orders.

  \medskip

  As in Case 1, we first note that $H_2$ is non empty, thanks to the
  choice of the order $M$ \eqref{defM} up to which we made the Taylor
  expansion. Furthermore,  the non-codominance property holds here too!

  \medskip
  
  Before proving that, we would like to comment on the time interval
  where we investigate the codominance property. In Case 1, all the
  terms in the expansion \eqref{sol} are multiples of the sole
  function $e^{\frac \tau 2}$. Comparing them at $\tau=0$ or on any
  subinterval of $[0, \infty)$ gives the same order. Here in Case 2,
  we have resonant terms, namely multiples of $(\san +\tau)^i e^{\frac
    \tau 2}$ with $i\in \m N$, and the comparison
  at $\tau=0$ or on a larger interval may be different, depending on
  the coefficient in front of the function and also on 
  the size of the interval. For that reason, we need to clearly fix
  some interval where we make the comparison. The most natural choice
  is simply the time interval of validity of the expansion
  \eqref{sol2}, namely, the interval where all the functions stay less
  than some fixed small $\delta_0>0$. More precisely, given some
  $(k,j,i)\in H_2$, we take $\tau \in [0, \tau_{k,j,i.n}]$ such that
  \[
 (\san+\tau_{k,j,i,n})^i e^{\frac {\tau_{k,j,i,n}} 2}e^{(1-\frac  k2)\san}A^{k-1}(k-j)\tilde
 C_{k,j,i}B_n^{k-1-j}=\delta_0.
\]
Since the function $\tau \mapsto (\san +\tau)^i e^{\frac \tau 2}$ is
increasing, we clearly have from \eqref{deftheta} that
$\tau_{k,j,i,n}\to \infty$ as $n\to \infty$, and 
\[
  \tau_{k,j,i,n}\sim [k-2+2\theta(k-1-j)]\san
 \mbox{ as }n\to \infty
\]
(note that the coefficient of $\san$ is positive, since $\theta\ge 0$
and $k\ge m \ge 6$).\\
Since $6\le m \le k\le M$ and $\theta\in [0, \frac 12]$, it follows that
$k-2+2\theta(k-1-j) \in [k-2, 2k-3]\subset [m-2, 2M-3]$, hence
\[
\tau_{k,j,in}\le (2M-2)\san
\]
for $n$ large enough. Thus, we will consider $\tau \in [0,(2M-2)\san]$. On
that interval, we see that the resonant function $(\san+\tau)^ie^{\frac
  \tau 2}$ is comparable to a pure exponential function, in the sense
that
\[
  \forall \tau \in [0, (2M-2)\san],\;\;
  \san^i e^{\frac \tau 2} \le (\san+\tau)^ie^{\frac \tau 2}
  \le (2M-2)^i  \san^i e^{\frac \tau 2}.
\]
Thus, the codominance property can be checked at $\tau =0$, as we
did in Case 1. 

\medskip

Consider then two different $(k,j,i)$ and
  $(k',j',i')$ in $H_2$, and let us show that either
  $(\san+\tau)^i e^{\frac \tau 2}e^{(1-\frac  k2)\san}B_n^{k-1-j}$ or 
 $(\san+\tau)^{i'} e^{\frac \tau 2} e^{(1-\frac {k'}2)\san}B_n^{k'-1-j'}$
 dominates the other, for $n$ large and $\tau=0$, which is legitimate,
 from the reduction we have just proved above.
 Note first from
\eqref{deftheta}
 that
 \[
   \san^i e^{(1-\frac  k2)\san}B_n^{k-1-j}
   =\san^i e^{(1-\frac  k2-(k-1-j)\theta+o(1)))\san},
 \]
 as $n\to \infty$,
 with a similar estimate for $(k',j',i')$.\\
If $(k,j)\neq (k',j')$, then $(k,k-j)\neq (k',k'-j')$, hence  $1-\frac
k2-(k-1-j)\theta\neq  1-\frac  {k'}2-(k'-1-j')\theta$, since $\theta$
is not rational.
Taking $n$ large enough, we
see that one term dominates the other.\\
Now, if $(k,j)=(k',j')$, then $i\neq i'$, and the two terms are
different by the power of $\san$ (their ratio is exactly $\san^{i-
  i'}$), and this implies that
one term dominates the other. Thus, we see that the non-codominance property
holds in Case 2 too. \\
Since the set $H_2$ is finite and non empty, we may consider $(\bar k,
\bar j, \bar i)\in H_2$ such that
\[
v_{b,1,0}(\san+\tau) \sim (\san+\tau)^{\bar i}e^{\frac \tau 2}(\bar k-\bar
j)\tilde C_{\bar
  k,\bar j, \bar i}A^{\bar k-1} B_n^{\bar k-1-j}e^{(1-\frac
    {\bar k}2)\san}
\]
as $n\to \infty$, with $\tilde C_{\bar k, \bar j, \bar i}\neq 0$.
From the growth factor $e^{\frac \tau 2}$, the coordinate
$v_{b,1,0}(\san+\tau)$ will grow, and a contradiction follows as usual.

  \medskip
  
  Of course, our argument in Cases 1 and 2 is formal, however, it can be made rigorous as
  usual, like we did at the end of
Part 1 of Subsection \ref{P3} in the case where $m=4$.
  This finishes the proof of item (i) in Proposition \ref{proptheta}.

\bigskip

\noindent - \textit{Proof of item
(ii)
  of Proposition
  \ref{proptheta}}.

The result comes from a small modification of the argument of the
proof of item (i). Assume by contradiction that $\theta$ defined in
\eqref{deftheta} is zero. All the argument of
the proof of item (i)
holds here, and
we naturally fall in Case 1. In particular \eqref{sol} holds and the
finite set $H_1$ defined in \eqref{defh1} is non empty. It remains
just to check the non dominance property stated in Lemma
\ref{lemnocodom}. Let us consider $(k,j)$ and
$(k',j')$ with $C_{k,j}C_{k',j'}\neq 0$, $m\le k,k'\le 2m-3$, $0\le j\le k-1$, $0\le
j'\le k'-1$, and  $(k,j) \neq(k',j')$, and show that either
$(k-j)C_{k,j}A^{k-1} B_n^{k-1-j}e^{(1-\frac  k2)\san}$ or
$(k'-j')C_{k',j'}A^{k'-1} B_n^{k'-1-j'}e^{(1-\frac {k'}2)\san}$
dominates the other .\\
If $k\neq k'$, then this is clear from \eqref{thomas}.\\
If $k=k'$, then, recalling that $(k,j) \neq(k',j')$, we necessarily
see that $j\neq j'$, hence the power of $B_n$ is not the same in the
two terms. Since $B_n\to 0$ from
\eqref{sizebn0},
one term dominates the other. Thus, Lemma
\ref{lemnocodom} holds here too, and one can carry on the argument of
the proof of item (i)
to derive that \eqref{cesaria} holds, which yields a
contradiction from the exponential factor. Thus, $\theta \neq 0$.

\bigskip

- \textit{Proof of item
(iii)
  of Proposition \ref{proptheta}}.
  
In this step, we explain how to derive item
(iii)
of Proposition \ref{proptheta} from the proof of item (i).

 As announced earlier, our argument is a small adaptation of the
 argument already used
 for the proof of item (i).
The key
idea for the adaptation was already mentioned in the remark following the proof
of Lemma \ref{lemnocodom}: having a non rational $\theta$ is a too
strong condition to guarantee non-codominance.
according to that remark and to the two cases mentioned in the proof
of item (i), we immediately see that $\theta$ should avoid the
following two sets, in order for the contradiction argument to work:
  \begin{align}
E_1= [0, \frac{m-2}{2(m-1)}]\cap & \{ \frac{k'-k}{2[(k-j)-(k'-j')]} \;\; | \;\;  m\le
                                                                                      k,k'\le 2m-3,\label{defE1}\\
    &0\le j\le k-1,\;\;0\le j'\le k'-1,
    (k,j)\neq (k',j') \mbox{ and }k-j\neq k'-j'
                 \}.\nonumber\\
E_2= [\frac{m-2}{2(m-1)}, \frac 12]\cap & \{ \frac{k'-k}{2[(k-j)-(k'-j')]} \;\; | \;\; m\le
                                                                                      k,k'\le 2m-2,\label{defE2}\\
    &0\le j\le k-1,\;\;0\le j'\le k'-1,
    (k,j)\neq (k',j') \mbox{ and }k-j\neq k'-j'
                 \}.\nonumber
  \end{align}
  This concludes the proof of Proposition \ref{proptheta}.
  \end{proof}

\bigskip

\textbf{Step 3: Polynomial corrections to the exponential decay of $B_n$}.

In Step 2,
we showed that
\begin{equation}\label{chafik}
  B_n = e^{-\theta \san}\varphi_n\mbox{ where }\varphi_n=e^{o(\san)} \mbox{ as }n\to \infty,
\end{equation}
for some $\theta\in (0,\frac 12] \cap \m Q$ enjoying only a finite
number of values. As one recalls from the
proof, our starting point in the proof lays in the
Taylor expansion of the solution $v_0(y,s)$ provided by Proposition
\ref{propexp}. As that proposition allows the existence of
``resonant'' terms, of the type $s^i e^{(1-\frac k2)s}$, it is natural
  to expect that the $o(\san)$ in \eqref{chafik} is of order $\alpha \log
  \san$ for some $\alpha \in \m R$, resulting in $\varphi_n \sim
  \san^\alpha$. This is precisely the aim of this
step.
  Let us then
  assume that for up some subsequence denoted the same, we have
  \begin{equation}\label{amira}
\frac{\log \varphi_n}{\log \san}\to \alpha \in [-\infty, +\infty].
\end{equation}
We claim the following:
\begin{prop}\label{propalpha}$ $\\
  (i) If $\theta \in [0, \frac{m-2}{2(m-1)})$, then $\alpha =0$.\\
  (ii) If $\theta \in [\frac{m-2}{2(m-1)}, \frac 12]$, then $\alpha \in E_3$
  defined
  by
\begin{align}\label{defE3}
  E_3\equiv \left\{\frac{i'-i}{k-j - (k'-j')}\;|\;m\le k,k'\le M,\;
  0\le j \le k-1,\; 0 \le j' \le k'-1,\; \right.\\
  \left.0\le i \le i_k,\;
  0\le i'\le i_{k'}, (k,j,i)\neq (k',j',i')\mbox{ and } k-j\neq k'-j'\right\},\nonumber
\end{align}
where $M$ is defined in \eqref{defM} and $i_k$ right after \eqref{defh2}.
\end{prop}
\begin{proof}
%
  We proceed in 2 steps:\\
  - We first show that $\alpha$ is finite.\\
  - Then, we show that $\alpha$ enjoys a finite number of rational values.

   \bigskip


- \textit{Proof of the fact that $\alpha$ is finite}.

 Let us
assume by contradiction that
\begin{equation}\label{eitheror0}
  \mbox{either }\frac{\log \varphi_n}{\log \san} \to -\infty\mbox{ or
  } \frac{\log \varphi_n}{\log \san} \to \infty
  \mbox{ as }n\to \infty.
\end{equation}
for a subsequence still denoted the same. In particular, this implies
that
\begin{equation}\label{eitheror}
  \mbox{either }\varphi_n \to 0 \mbox{ or }\varphi_n \to \infty,
   \mbox{ as }n\to \infty.
 \end{equation}
 Our idea is to follow the pattern of
the proof of item (i) of Proposition \ref{proptheta},
 where we proved
that $\theta$ is rational. Naturally, we consider the two cases
mentioned in that step.

\medskip

\textbf{Case 1: $\theta \in [0, \frac{m-2}{2(m-1)})$}.
Even though $\theta$ is rational this time, the expansion \eqref{sol}
remains valid with the same finite set $H_1$ defined in \eqref{defh1}, which is non empty, for
the same reason. If we can show the non-codominance property as in
Lemma \ref{lemnocodom}, then, we are done. Let us then prove that
lemma, in this new setting. Of course, we need a different argument,
since $\theta$ is rational this time, and the issue concerns the
following order term, involving the limit $\alpha$ defined in \eqref{amira}. Consider then two terms  in the
expansion \eqref{sol} of
  $v_{b,1,0}(\san + \tau)$, say
  $(k-j)C_{k,j}A^{k-1} B_n^{k-1-j}e^{(1-\frac
    k2)\san}$ and
  $(k'-j')C_{k',j'}A^{k'-1} B_n^{k'-1-j'}e^{(1-\frac
    {k'}2)\san}$ with $C_{k,j}C_{k',j'}\neq 0$, $m\le k,k'\le 2m-3$, $0\le j\le k-1$, $0\le
  j'\le k'-1$, and  $(k,j) \neq(k',j')$.
  By definition \eqref{chafik} of $\varphi_n$, we see that
  \begin{equation}\label{des}
    (k-j)C_{k,j}A^{k-1} B_n^{k-1-j}e^{(1-\frac  k2)\san}
= (k-j)C_{k,j}A^{k-1}  e^{(1-\frac
  k2-(k-1-j)\theta))\san}\varphi_n^{k-1-j},
\end{equation}
with a similar expansion for $(k',j')$.\\
If $1-\frac k2 -(k-1-j) \theta \neq 1-\frac {k'}2 -(k'-1-j') \theta$,
recalling that $\varphi_n = e^{o(\san)}$ as stated in \eqref{chafik},
we see that one of the two terms dominates the other.\\
Now, if $1-\frac k2 -(k-1-j) \theta = 1-\frac {k'}2 -(k'-1-j')\theta$
(and this may occur since $\theta$ is rational),
recalling that $(k,j)\neq (k',j')$ we necessarily have
$k-1-j\neq k'-1-j'$. In other words, the power of $\varphi_n$ is not
the same in the two terms. Since $\varphi_n\to 0$ or $\varphi_n \to
\infty$, as stated in \eqref{eitheror}, we see from the description
\eqref{des} that one of the two terms has to
dominate the other. Thus,
the non-codominance property holds,
and a contradiction follows as in Case 1 of
the proof of item (i) of Proposition \ref{proptheta}.

\medskip

\textbf{Case 2: $\theta \in [\frac{m-2}{2(m-1)}, \frac 12]$}. Again,
the argument of Case 2 in
the proof of item (i) of Proposition \ref{proptheta}
  holds, and we may derive an
expansion of $v_{b,1,0}(\san +\tau)$ like in \eqref{sol2}, with
``resonant'' terms. The finite set $H_2$ is non empty, for the same
reason. If we can show the non-codominance property as before, then we
are done.\\
Consider then two different $(k,j,i)$ and
  $(k',j',i')$ in $H_2$, and let us show that either
  $(\san+\tau)^i e^{\frac \tau 2}e^{(1-\frac  k2)\san}B_n^{k-1-j}$ or 
 $(\san+\tau)^{i'} e^{\frac \tau 2} e^{(1-\frac {k'}2)\san}B_n^{k'-1-j'}$
 dominates the other, for $n$ large. From the reduction we showed in
 Case 2 of
the proof of item (i) of Proposition \ref{proptheta}
 above, it is enough to check the dominance at
 $\tau=0$. 
Note first from \eqref{chafik} that
 \begin{equation}\label{elham}
   \san^i  e^{(1-\frac  k2)\san}B_n^{k-1-j}
   =\san^i  e^{(1-\frac  k2-(k-1-j)\theta))\san}\varphi_n^{k-1-j},
 \end{equation} 
 as $n\to \infty$,
 with a similar estimate for $(k',j',i')$.\\
  If $1-\frac  k2-(k-1-j)\theta\neq  1-\frac  {k'}2-(k'-1-j')\theta$, 
 recalling that $\varphi_n = e^{o(\san)}$ from \eqref{chafik},
 we see that one term dominates the other.\\
  Assume then that  $1-\frac  k2-(k-1-j)\theta=  1-\frac
  {k'}2-(k'-1-j')\theta$.\\
If $k-j=k'-j'$, then $k=k'$, and since $(k,j,i)\neq (k',j',i')$, it
follows that $i\neq i'$. From the expression \eqref{elham}, we see
that the powers of $\san$ are different, hence, one term
dominates the other.\\
Now, if $k-j\neq k'-j'$, making the ratio between the two terms, we
find $\san^{i-i'}\varphi_n^{k-j-(k'-j')}$. Using
\eqref{chafik} and \eqref{eitheror}, 
we see that one term dominates the other.

  \medskip
  
  Of course, our argument in Cases 1 and 2 is formal, however, it can be made rigorous as
  usual, like we did at the end of
Part 1 of Subsection \ref{P3} in the case $m=4$.
Thus, we
have just proved 
that the parameter $\alpha$ defined in
\eqref{amira} is finite.

\bigskip


- \textit{Conclusion of the proof of Proposition \ref{propalpha}}.

In the previous step, we assumed that $\alpha = \pm \infty$ and reached a
contradiction. In fact, a careful check reveals that the contradiction
can be reached in Case 1, for any $\alpha \neq 0$, whereas in Case 2,
we simply need $\alpha$ to avoid the set
$E_3$ defined in \eqref{defE3}.
This concludes the proof of Proposition \ref{propalpha}.
\end{proof}
  

\textbf{
Step 4:
  One further refinement in the behavior of $B_n$}

So far, thanks to Propositions \eqref{proptheta} and
\eqref{propalpha}, we have proved that
\begin{equation}\label{defpsin}
  B_n = \exp(-\theta \san)\san^\alpha\psi_n
  \mbox{ for some }\psi_n = \san^{o(1)}
\end{equation}
as $n\to \infty$.
This is the aim of this step:
\begin{prop} \label{climen} Up to a subsequence,
$\psi_n$ converges to some $L>0$, where $L$ enjoys a finite number of
values, all solutions of polynomials whose coefficients depend on
the coefficients that arise in the Taylor expansion of 
$v_0(y,s)$.
\end{prop}
\begin{proof}
  Here again,
  we crucially use the geometric transformation
  introduced in
  Step 2
  of Section \ref{sectaylor}.
We proceed in 2 steps:\\
- In Step (i),
we show that $\psi_n$ is bounded away from $0$ and from infinity.\\
 - In Step (ii),
we show that up to a subsequence,
$\psi_n$ converges to some $L>0$,  which is a solution of a polynomial
whose coefficients depend on $C_{k,j}$.

\bigskip

\textit{
  - Step (i):
  $\psi_n$ is bounded away from $0$ and from infinity}.

We proceed by contradiction, and assume that for a subsequence (still
denoted the same), we have
\begin{equation}\label{eitheror2}
\psi_n \to 0\mbox{ or } \psi_n \to \infty\mbox{ as }n\to \infty.
\end{equation}
As before, we follow the strategy of
the proof of item (i) of Proposition \ref{proptheta}, with its two cases.\\
Starting by Case 1, where $\theta \in [0, \frac{m-2}{2(m-1)})$, we
still see that \eqref{sol}
holds,
with $H_1$ which is still a non empty
finite set. It remains only to prove the non codominance
property.
Consider then two terms  in the
expansion \eqref{sol} of
  $v_{b,1,0}(\san + \tau)$, say
  $(k-j)C_{k,j}A^{k-1} B_n^{k-1-j}e^{(1-\frac
    k2)\san}$ and
  $(k'-j')C_{k',j'}A^{k'-1} B_n^{k'-1-j'}e^{(1-\frac
    {k'}2)\san}$ with $C_{k,j}C_{k',j'}\neq 0$, $m\le k,k'\le 2m-3$, $0\le j\le k-1$, $0\le
  j'\le k'-1$, and  $(k,j) \neq(k',j')$. Using \eqref{defpsin}, and
  recalling that $\alpha=0$ from Proposition \ref{propalpha}, we see that
  \begin{equation}\label{des2}
    (k-j)C_{k,j}A^{k-1} B_n^{k-1-j}e^{(1-\frac  k2)\san}
= (k-j)C_{k,j}A^{k-1}  e^{(1-\frac
  k2-(k-1-j)\theta)\san}\psi_n^{k-1-j},
\end{equation}
with a similar expansion for $(k',j')$.\\
If $1-\frac k2 -(k-1-j) \theta \neq 1-\frac {k'}2 -(k'-1-j') \theta$,
recalling that $\psi_n = \san^{o(1)}=e^{o(\log \san)}$ as stated in \eqref{defpsin},
we see that one of the two terms dominates the other.\\
Now, if $1-\frac k2 -(k-1-j) \theta = 1-\frac {k'}2 -(k'-1-j')
\theta$, recalling that $(k,j)\neq (k',j')$ we necessarily have
$k-1-j\neq k'-1-j'$. In other words, the power of $\psi_n$ is not
the same in the two terms. Since $\psi_n\to 0$ or $\psi_n \to
\infty$, as stated in \eqref{eitheror2}, we see from the description
\eqref{des2}
that one of the two terms has to
dominate the other. Thus,
co-dominance
holds in this context, and a
contradiction follows as in Case 1 of
the proof of item (i) of Proposition \ref{proptheta}.\\
Now, moving to Case 2, where $\theta\in [\frac{m-2}{2(m-1)}, \frac
12]$, we may consider two terms in the expansion \eqref{sol2} of
$v_{b,1,0}(\san + \tau)$, say
$(\san+\tau)^i e^{\frac \tau 2}e^{(1-\frac  k2)\san}B_n^{k-1-j}$ and
 $(\san+\tau)^{i'} e^{\frac \tau 2} e^{(1-\frac
   {k'}2)\san}B_n^{k'-1-j'}$, where $(k,j,i)$ and $(k',j',i')$ belong to
 $H_2$ defined in \eqref{defh2}, and prove that one dominates the
 other.
 From the reduction we did in Case 2 of
the proof of item (i) of Proposition \ref{proptheta},
 we may check the
 dominance only at $\tau=0$.
 Using \eqref{defpsin}, we see that
 \[
   \san^i e^{(1-\frac  k2)\san}B_n^{k-1-j}
   = e^{(1-\frac
     k2-(k-1-j)\theta)\san}\san^{i+\alpha(k-1-j)} \psi_n^{k-1-j}
 \]
 as $n\to \infty$, with a similar estimate for $(k',j',i')$.\\
 If $1-\frac k2 -(k-1-j) \theta \neq 1-\frac {k'}2 -(k'-1-j') \theta$,
recalling that $\psi_n = \san^{o(1)}=e^{o(\log\san)}$ as stated in \eqref{defpsin},
we see that one of the two terms dominates the other.\\
Assume then that $1-\frac k2 -(k-1-j) \theta = 1-\frac {k'}2 -(k'-1-j')
\theta$.\\
If $i+\alpha(k-1-j)\neq i'+(k'-1-j')\alpha$, then, using again
the fact
that
$\psi_n = \san^{o(1)}=e^{o(\log\san)}$,  we see that the power of
$\san$ is different between the two terms, hence one term dominates the other.\\
Now, if $i+\alpha(k-1-j)=i'+(k'-1-j')\alpha$, then, necessarily $k-1-j\neq
k'-1-j'$, otherwise, $(k,j,i)=(k',j',i')$. Therefore, the power of $\san$ is
the same in the two terms,
unlike the power of $\psi_n$. Recalling that $\psi_n \to 0$ or $\psi_n
\to \infty$, we see that one term dominates the other.

\bigskip

\textit{Step (ii):
  $\psi_n$ converges (up to a subsequence) }

From Step (i),
we may assume that for some
subsequence (still denoted the same), we have
\begin{equation}\label{cv}
\psi_n \to A^{2\theta}  L\mbox{ as }n\to \infty
\end{equation}
for some $L>0$ (we write the limit in \eqref{cv} as $A^{2\theta}  L$
and not $L$, since in this form, $L$ will be shown to be a solution of
a polynomial whose coefficients are independant  of $A$). \\
We will
distinguish two cases as in
the proof of item (i) of Proposition \ref{proptheta}.\\
\textbf{Case 1: $\theta \in [0, \frac{m-2}{2(m-1)})$}. Note that
$\alpha=0$ by Proposition \ref{propalpha}. In this case, we
see that \eqref{sol} still holds, with $H_1$ which is still a non empty
finite set. However, this time, we will have codominance, as we show
in the following. Since the number of terms is finite in the expansion
\eqref{sol}, we may consider $(\bar k, \bar j)\in H_1$ defined in
\eqref{defh1} such that the term corresponding to this parameter
dominates all the others. Using \eqref{sol} and \eqref{defpsin}, this
term reads
\begin{equation}\label{diams}
e^{(1-\frac{\bar k}2) \san} A^{\bar k-1} (\bar k - \bar j) C_{\bar
  k,\bar j} B_n^{\bar k -1 - \bar j} =
A^{\bar k-1} C_{\bar k,  \bar j} (\bar k - \bar j)
e^{(1-\frac{\bar k}2-(\bar k-1-\bar j)\theta)\san} \psi_n^{\bar k -1 - \bar j}.
\end{equation}
As we wrote earlier, we expect here to have
codominance. Let us then characterize the set $\bar E$ of all $(k,j)\in H_1$ such that the
corresponding term is of the same size as the term corresponding to
$(\bar k, \bar j)$. From \eqref{diams} and \eqref{cv}, this means that
\begin{equation}\label{nana}
1-\frac{\bar k}2-(\bar k-1-\bar j)\theta = 1-\frac k2-(k-1-j)\theta.
\end{equation}
This way, we may keep only the dominant terms in \eqref{sol}, namely
those coming from $\bar E$, and write
\[
 v_{b,1,0}(\san+\tau) = e^{\frac \tau 2}\left(
    \sum_{(k,j)\in \bar E}e^{(1-\frac  k2)\san}A^{k-1}
    (k-j)C_{k,j}B_n^{k-1-j}
                        +O(\san e^{(2-m)\san})\right).
\]
Using \eqref{diams} and \eqref{nana}, together with the convergence
\eqref{cv},
we derive that
\begin{align*}
v_{b,1,0}(\san+\tau) &= e^{\frac \tau 2}  e^{(1-\frac{\bar
      k}2-(\bar k-1-\bar j)\theta)\san}
  \left( \sum_{(k,j)\in \bar E} A^{k-1}
    (k-j)C_{k,j}(A^{2\theta} L)^{k-1-j} +o(1) \right)\\
   & =  e^{\frac \tau 2}  e^{(1-\frac{\bar
        k}2-(\bar k-1-\bar j)\theta)\san}
     A^{\bar k-1+2\theta(\bar k-1-\bar j)}
  \left( \sum_{(k,j)\in \bar E}
    (k-j)C_{k,j}L^{k-1-j} +o(1) \right)
\end{align*}
as $n\to \infty$.
From the growth factor $e^{\frac \tau 2}$, this implies that
  \[
\sum_{(k,j)\in \bar E}  (k-j)C_{k,j}L^{k-1-j} =0.
\] 
Since $L\neq  0$ and $C_{k,j}\neq 0$, for all $(k,j)\in \bar E$
(remember that $\bar E\subset H_1$ defined in \eqref{defh1}), this
sum
contains at least two terms, and this is precisely the desired polynomial relation. Remember that $\bar E\neq\emptyset$, since it contains $(\bar k, \bar j)$, and that $C_{k,j}\neq
0$, for any $(k,j) \in \bar E$, since $\bar E\subset H_1$ defined in
\eqref{defh1}. Note that the degree of this
polynomial is bounded by $k-1\le 2m-4$ by definition \eqref{defh1} of
$H_1$, hence, we have at most $2m-4$ possible values for $L$.\\
\textbf{Case 2: $\theta \in [\frac{m-2}{2(m-1)}, \frac
  12]$}. We are then in the framework of Case 2 of
the proof of item (i) of Proposition \ref{proptheta}
above. In
particular, the finite set $H_2$
is still non empty, for the same reason. However,
we may have codominance in this
context. Since the number of terms is finite in the expansion
\eqref{sol2}, we may consider $(\bar k, \bar j, \bar i)\in H_2$
defined in \eqref{defh2} such that the term corresponding to this
parameter dominates all the others. Following what we wrote in Case 2
of
the proof of item (i) of Proposition \ref{proptheta}
above, we may just take $\tau=0$ to discuss codominance
issues. 
Using \eqref{sol2} and
\eqref{defpsin}, this term reads as follows, when $\tau=0$:
\begin{align}
&\san^{\bar i}e^{(1-\frac {\bar k}2)\san}A^{\bar k-1}
    (\bar k-\bar j)\tilde C_{\bar k,\bar j,\bar i}B_n^{\bar k-1-\bar
      j}\label{diams2}\\
    =&  e^{(1-\frac {\bar k}2-\theta(\bar k - 1-
      \bar j))\san}\san^{\bar i+\alpha(\bar k -1- \bar j)} A^{\bar k-1}
    (\bar k-\bar j)  \tilde C_{\bar k,\bar j,\bar i} \psi_n^{\bar k
      -1- \bar j}.\nonumber
\end{align}
Let us then
characterize the set $\bar E$ of all $(k,j,i)\in H_2$ such that the
corresponding term in \eqref{sol2} is of the same size as the term
corresponding to $(\bar k, \bar j, \bar i)$. From \eqref{diams2} and
\eqref{cv}, we see that we need to have
\begin{equation}\label{nana2}
  1-\frac {\bar k}2 - \theta(\bar k - 1 - \bar j) = 1- \frac k2 - \theta(k-1-j)
  \mbox{ and }
  \bar i + \alpha(\bar k -1- \bar j)= i+\alpha(k-1-j).
\end{equation}
This way, we keep only the dominant terms in \eqref{sol2}, and write
\begin{align*}
  v_{b,1,0}(\san+\tau) = e^{\frac \tau 2}\left(
    \sum_{(k,j,i)\in \bar E}\right.&(\san+\tau)^i e^{(1-\frac  k2)\san}A^{k-1}
    (k-j)\tilde C_{k,j,i}B_n^{k-1-j}\nonumber\\
    &\left.+O(\san^\gamma e^{(\frac{1-M}2)\san})\right). 
\end{align*}
Using \eqref{diams2} and \eqref{nana2}, together with \eqref{cv}, we
write for all $\tau \in [0, (2M-2)\san]$ (note that this interval
refers to our discussion in Case 2
of the proof of item (i) of Proposition \ref{proptheta}):
\begin{align*}
 v_{b,1,0}(\san+\tau) = &e^{\frac \tau 2}e^{(1-\frac{\bar k}2-
   \theta(\bar k -1- \bar j))\san}\san^{\bar i + \alpha(\bar k -1-
   \bar j)} A^{\bar k -1 +2\theta(\bar k -1 - \bar j)}\\
&\times \left(\sum_{(k,j,i)\in \bar E}\left(1+\frac \tau\san\right)^i(k-j) \tilde C_{k,j,i}L^{k-1-j}+o(1)\right).
\end{align*}
Now, for any $\delta_0>0$, we consider $\tau_n(\delta_0)>0$ such that
\[
e^{\frac \tau 2}e^{(1-\frac{\bar k}2-
   \theta(\bar k -1- \bar j))\san}\san^{\bar i + \alpha(\bar k -1-
   \bar j)} A^{\bar k -1 +2\theta(\bar k -1 - \bar j)}=\delta_0.
\]
Clearly, it holds that
\[
  \tau_n(\delta_0) \sim [\bar k - 2+\theta(\bar k -1 - \bar j)]\san
  \mbox{ as }n\to \infty.
\]
Therefore, we see that
\[
v_{b,1,0}(\san+\tau_n(\delta_0)) = \delta_0\left(\sum_{(k,j,i)\in \bar E}\left([\bar k - 1 +\theta(\bar k -1 - \bar j)\right)^i(k-j) \tilde C_{k,j,i}L^{k-1-j}+o(1)\right).
\]
This forces the coefficient of $\delta_0$ to be zero:
\[
\sum_{(k,j,i)\in \bar E}\left([\bar k - 1 +\theta(\bar k -1 - \bar j)\right)^i(k-j) \tilde C_{k,j,i}L^{k-1-j}=0,
\]
otherwise,
$v_{b,1,0}(\san+\tau_n(\delta_0))$ will be large as in \eqref{hadaf}, which leads to a contradiction.\\
Since $L\neq  0$ and $\tilde C_{k,j,i}\neq 0$, for all $(k,j,i)\in
\bar E$ (remember that $\bar E\subset H_2$ defined in \eqref{defh2}),
this identity contains at least two terms, and this is precisely the
desired polynomial relation.\\
This concludes the proof of
Proposition
\ref{climen}
and finishes Part 1 dedicated to the super-quadratic case.
\end{proof}

\textbf{
Part 2:
 The quadratic regime}

In this part, we consider the quadratic regime, where
\begin{equation}\label{quad0}
\aa \sim L \bb^2 \mbox{ as }n\to \infty,
\end{equation}
for some $L>0$, and show that $L$ satisfies a polynomial equation whose coefficients are given by the Taylor expansion of the solution. 

\medskip

Note first from the notation \eqref{an2} and the definition \eqref{defbn-} of $B_n$ that we have
\[
B_n \sim LA e^{-\frac{\san}2}\mbox{ as }n\to \infty.
\]
The proof is in fact a simple adaptation of our argument in the
super-quadratic regimes, given in
Part 1.
Let us then follow
that part
step by step, and see what changes.

\medskip

In
Part 1, 
the outcome of Step 1 is the following:
\begin{equation}\label{asperger}
C_{m+1,m}+LC_{m,m-2}=0.
\end{equation}
If $C_{m,m-2}\neq 0$, then, we have our polynomial and we are done (in
fact, we have more, in the sense that $L$ enjoys only one value : $-
C_{m+1,m}/C_{m,m-2}$).\\
If  $C_{m,m-2}= 0$, then we also have $C_{m,m-3}=0$, because the
multilinear form in
\eqref{defmult}
is nonpositive. In other words,
we have exactly the same conclusions as in Step 1 of
Part 1
in the
super-quadratic case (in fact, we have more, since $C_{m+1,m}=0$ from
\eqref{asperger}). For short, we can carry on all the next steps
of the super-quadratic case up to
to the end of Part 1,
and see that $L$
satisfies a polynomial equations. For the reader's convenience, we
would like to mention that hypothesis \eqref{quad0} makes many steps
either non relevant or trivial:\\
- In Step 2 of Part 1,
$\theta=\frac 12$
which makes Proposition \ref{proptheta} non relevant. Moreover, 
\eqref{quad0} is stronger than \eqref{deftheta}.\\
- In Step 3,
estimate \eqref{quad0} is
stronger than \eqref{chafik}, and $\alpha$ defined in \eqref{amira} is
zero. Accordingly,
Proposition \ref{propalpha} is
non relevant.\\
- As for
Step 4,
again, estimate \eqref{quad0} is stronger than \eqref{defpsin},
and the first assertion of Proposition \ref{climen} is clear
($\psi_n\to L$ as $\to \infty$).\\
- In conclusion, only the second assertion of Proposition \ref{climen}  
remains
relevant, and provides us with the polynomial relation for $L$.

\bigskip

\textbf{
Part 3:
Conclusion of the proof of Theorem
\ref{cor0m}}

From the study of the super-quadratic case in
Part 1,
and
also the quadratic case in
Part 2,
we see that either we are in the
subquadratic case $\aa = o(\bb)^2$, or
\[
  B_n \sim L A^{2\theta} e^{-\theta \san}\san^\alpha \mbox{ as }
  n\to \infty,
\]
where $\theta$ and $\alpha$ enjoy only a finite set of rational
values, and $L>0$ is a solution of a polynomial equations whose
coefficients depend on the Taylor expansion of the solution. By
definitions \eqref{an2} and \eqref{defbn-} of $\san$ and $B_n$, we see that
\[
\aa \sim L\bb^{2\theta+1}|2\log \bb|^{\alpha},
\]
with $2\theta+1\in (0,1]$, which is the desired estimate in
Theorem
\ref{cor0m}
The set where $2\theta+1$ lives directly
follows from item
(iii)
in Proposition \ref{proptheta}.

\medskip

This concludes the proof of Theorem
\ref{cor0m}.
\end{proof}

\def\cprime{$'$} \def\cprime{$'$}


\begin{thebibliography}{10}

\bibitem{CMjde89}
X.~Y. Chen and H.~Matano.
\newblock Convergence, asymptotic periodicity, and finite-point blow-up in
  one-dimensional semilinear heat equations.
\newblock {\em J. Differential Equations}, 78(1):160--190, 1989.

\bibitem{FKcpam92}
S.~Filippas and R.~V. Kohn.
\newblock Refined asymptotics for the blowup of $u\sb t-{\Delta} u=u\sp p$.
\newblock {\em Comm. Pure Appl. Math.}, 45(7):821--869, 1992.

\bibitem{FLihp93}
S.~Filippas and W.~X. Liu.
\newblock On the blowup of multidimensional semilinear heat equations.
\newblock {\em Ann. Inst. H. Poincar\'e Anal. Non Lin\'eaire}, 10(3):313--344,
  1993.

\bibitem{GNZans17}
T.~Ghoul, V.T. Nguyen, and H.~Zaag.
\newblock Refined regularity of the blow-up set linked to refined asymptotic
  behavior for the semilinear heat equation.
\newblock {\em Adv. Nonlinear Stud.}, 17(1):31--54, 2017.

\bibitem{GKcpam85}
Y.~Giga and R.~V. Kohn.
\newblock Asymptotically self-similar blow-up of semilinear heat equations.
\newblock {\em Comm. Pure Appl. Math.}, 38(3):297--319, 1985.

\bibitem{GKcpam89}
Y.~Giga and R.~V. Kohn.
\newblock Nondegeneracy of blowup for semilinear heat equations.
\newblock {\em Comm. Pure Appl. Math.}, 42(6):845--884, 1989.

\bibitem{GMSiumj04}
Y.~Giga, S.~Matsui, and S.~Sasayama.
\newblock Blow up rate for semilinear heat equations with subcritical
  nonlinearity.
\newblock {\em Indiana Univ. Math. J.}, 53(2):483--514, 2004.

\bibitem{HVcpde92}
M.~A. Herrero and J.~J.~L. Vel{\'a}zquez.
\newblock Blow-up profiles in one-dimensional, semilinear parabolic problems.
\newblock {\em Comm. Partial Differential Equations}, 17(1-2):205--219, 1992.

\bibitem{HVihp93}
M.~A. Herrero and J.~J.~L. Vel{\'a}zquez.
\newblock Blow-up behaviour of one-dimensional semilinear parabolic equations.
\newblock {\em Ann. Inst. H. Poincar\'e Anal. Non Lin\'eaire}, 10(2):131--189,
  1993.

\bibitem{KRZihp11}
S.~Khenissy, Y.~R{\'e}ba{\"{\i}}, and H.~Zaag.
\newblock Continuity of the blow-up profile with respect to initial data and to
  the blow-up point for a semilinear heat equation.
\newblock {\em Ann. Inst. H. Poincar\'e Anal. Non Lin\'eaire}, 28(1):1--26,
  2011.

\bibitem{MRSimrn20}
F.~{Merle}, P.~{Rapha\"el}, and J.~{Szeftel}.
\newblock {On strongly anisotropic type I blowup}.
\newblock {\em {Int. Math. Res. Not.}}, 2020(2):541--606, 2020.

\bibitem{MZcpam98}
F.~Merle and H.~Zaag.
\newblock Optimal estimates for blowup rate and behavior for nonlinear heat
  equations.
\newblock {\em Comm. Pure Appl. Math.}, 51(2):139--196, 1998.

\bibitem{MZma00}
F.~Merle and H.~Zaag.
\newblock A {L}iouville theorem for vector-valued nonlinear heat equations and
  applications.
\newblock {\em Math. Annalen}, 316(1):103--137, 2000.

\bibitem{Sjfa12}
R.~Schweyer.
\newblock Type {II} blow-up for the four dimensional energy critical semi
  linear heat equation.
\newblock {\em J. Funct. Anal.}, 263(12):3922--3983, 2012.

\bibitem{Vcpde92}
J.~J.~L. Vel{\'a}zquez.
\newblock Higher-dimensional blow up for semilinear parabolic equations.
\newblock {\em Comm. Partial Differential Equations}, 17(9-10):1567--1596,
  1992.

\bibitem{Vtams93}
J.~J.~L. Vel{\'a}zquez.
\newblock Classification of singularities for blowing up solutions in higher
  dimensions.
\newblock {\em Trans. Amer. Math. Soc.}, 338(1):441--464, 1993.

\bibitem{Zihp02}
H.~Zaag.
\newblock On the regularity of the blow-up set for semilinear heat equations.
\newblock {\em Ann. Inst. H. Poincar\'e Anal. Non Lin\'eaire}, 19(5):505--542,
  2002.

\bibitem{Zcmp02}
H.~Zaag.
\newblock One dimensional behavior of singular ${N}$ dimensional solutions of
  semilinear heat equations.
\newblock {\em Comm. Math. Phys.}, 225(3):523--549, 2002.

\bibitem{Zbeit00}
H.~Zaag.
\newblock Regularity of the blow-up set and singular behavior for semilinear
  heat equations.
\newblock In {\em Mathematics \& mathematics education (Bethlehem, 2000)},
  pages 337--347. World Sci. Publishing, River Edge, NJ, 2002.

\bibitem{Zdmj06}
H.~Zaag.
\newblock Determination of the curvature of the blow-up set and refined
  singular behavior for a semilinear heat equation.
\newblock {\em Duke Math. J.}, 133(3):499--525, 2006.

\end{thebibliography}

\noindent{\bf Address}:\\
CY Cergy Paris Universit\'e, D\'epartement de math\'ematiques, 
2 avenue Adolphe Chauvin, BP 222, 95302 Cergy Pontoise cedex, France.\\
\vspace{-7mm}
\begin{verbatim}
e-mail: merle@math.u-cergy.fr
\end{verbatim}
Universit\'e Sorbonne Paris Nord, Institut Galil\'ee, 
Laboratoire Analyse, G\'eom\'etrie et Applications, CNRS UMR 7539,
99 avenue J.B. Cl\'ement, 93430 Villetaneuse, France.\\
\vspace{-7mm}
\begin{verbatim}
e-mail: Hatem.Zaag@univ-paris13.fr
\end{verbatim}
\end{document}